\documentclass[11pt,a4paper]{article}

\usepackage[OT1]{fontenc}
\usepackage[english]{babel}
\usepackage{latexsym,amsmath,amsfonts,amssymb,amsthm,graphicx}
\usepackage[round]{natbib}
\usepackage[colorlinks,citecolor=blue,urlcolor=blue,linkcolor=blue]{hyperref}
\usepackage{xcolor}
\usepackage{enumitem}
\usepackage{rotating}
\usepackage{subcaption}
\usepackage{authblk}
\usepackage{orcidlink}

\numberwithin{equation}{section}
\newtheoremstyle{localrem}
	{5pt} 
	{5pt} 
	{\rm} 
	{} 
	{\bf} 
	{{\rm.}} 
	{.7em} 
	{} 

\theoremstyle{localrem}
\newtheorem{Definition}{Definition}[section]

\newtheoremstyle{localthm}
	{5pt} 
	{5pt} 
	{\sl} 
	{} 
	{\bf} 
	{{\rm.}} 
	{.7em} 
	{} 

\theoremstyle{localthm}

\newtheorem{Theorem}[Definition]{Theorem}
\newtheorem{Lemma}[Definition]{Lemma}

\DeclareFontFamily{U}{mathx}{\hyphenchar\font45}
\DeclareFontShape{U}{mathx}{m}{n}{
      <5> <6> <7> <8> <9> <10>
      <10.95> <12> <14.4> <17.28> <20.74> <24.88>
      mathx10
      }{}
\DeclareSymbolFont{mathx}{U}{mathx}{m}{n}
\DeclareFontSubstitution{U}{mathx}{m}{n}
\DeclareMathAccent{\widecheck}{0}{mathx}{"71}

\def\bs{\boldsymbol}

\def\A{\bs{A}}
\def\u{\bs{u}}
\def\x{\bs{x}}
\def\xtil{\tilde{\bs{x}}}

\def\DD{{\cal D}}
\def\LL{{\cal L}}

\def\PP{{\cal P}}

\def\hat{\widehat}
\def\check{\widecheck}

\def\Pr{\operatorname{\mathbb{P}}}

\def\thb{\bs{\theta}}
\def\thbtil{\tilde{\bs{\theta}}}

\def\N{\mathbb{N}}
\def\R{\mathbb{R}}
\def\Rup{\R_\uparrow}

\def\argmax{\mathop{\rm arg\,max}}
\def\argmin{\mathop{\rm arg\,min}}

\def\lest{\le_{\mathrm{st}}}
\def\lelr{\le_{\mathrm{lr}}}

\newcommand{\proglang}[1]{\textsf{#1}}
\newcommand{\pkg}[1]{\textsf{#1}}

\author[1,2]{Alexandre M\"osching\,\orcidlink{0000-0002-8270-3724}\thanks{alexandre.moesching@roche.com}}
\author[1]{Lutz D\"umbgen\,\orcidlink{0000-0003-0172-9285}\thanks{lutz.duembgen@unibe.ch}}
\affil[1]{University of Bern, Department of Mathematics and Statistics, Bern, Switzerland}
\affil[2]{F. Hoffmann-La Roche Ltd, Nonclinical Biostatistics, Basel, Switzerland}
\date{\today}
\title{Estimation of a Likelihood Ratio Ordered\\ Family of Distributions}

\begin{document}
\maketitle

\begin{abstract}
Consider bivariate observations $(X_1,Y_1), \ldots, (X_n,Y_n) \in \R\times \R$ with unknown conditional distributions $Q_x$ of $Y$, given that $X = x$. The goal is to estimate these distributions under the sole assumption that $Q_x$ is isotonic in $x$ with respect to likelihood ratio order. If the observations are identically distributed, a related goal is to estimate the joint distribution $\LL(X,Y)$ under the sole assumption that it is totally positive of order two. An algorithm is developed which estimates the unknown family of distributions $(Q_x)_x$ via empirical likelihood. The benefit of the stronger regularization imposed by likelihood ratio order over the usual stochastic order is evaluated in terms of estimation and predictive performances on simulated as well as real data.



\paragraph{Acknowledgements:}
The authors are grateful to Johanna Ziegel, Alexander Jordan and Tilmann Gneiting for stimulating discussions and useful hints. We also thank a reviewer for constructive comments. This work was supported by Swiss National Science Foundation.

\end{abstract}


\section{Introduction}
%
Consider a univariate regression setting with observations $(X_1,Y_1)$, $(X_2,Y_2)$, \ldots, $(X_n,Y_n)$ in $\mathfrak{X} \times \R$, where $\mathfrak{X}$ is an arbitrary real set. We assume that conditional on $\bs{X} := (X_i)_{i=1}^n$, the observations $Y_1,Y_2,\ldots,Y_n$ are independent with distributions $\LL(Y_i \,|\, \bs{X}) = Q_{X_i}$, where the distributions $Q_x$, $x \in \mathfrak{X}$, are unknown. The goal is to estimate the latter under the sole assumption that $Q_x$ is isotonic in $x$ in a certain sense. That means, if $(X,Y)$ denotes a generic observation, the larger (or smaller) the value of $X$, the larger (or smaller) $Y$ tends to be. An obvious notion of order would be the usual stochastic order, which states that $Q_{x_1} \lest Q_{x_2}$ whenever $x_1 \le x_2$, that is, $Q_{x_1}((-\infty,y]) \ge Q_{x_2}((-\infty,y])$ for all $y\in\R$. This concept has been investigated and generalized by numerous authors, see \cite{Moesching_2020}, \cite{Henzi_Ziegel_Gneiting_2021} and the references cited therein. The latter paper illustrates the application of isotonic distributional regression in weather forecasting, and \cite{Henzi_etal_2021} use it to analyze the length of stay of patients in Swiss hospitals.

The present paper investigates a stronger notion of order, the so-called likelihood ratio order. The usual definition is that for arbitrary points $x_1 < x_2$ in $\mathfrak{X}$, the distributions $Q_{x_1}$ and $Q_{x_2}$ have densities $g_{x_1}$ and $g_{x_2}$ with respect to some dominating measure such that $g_{x_2}/g_{x_1}$ is isotonic on the set $\{g_{x_1} + g_{x_2} > 0\}$, and this condition will be denoted by $Q_{x_1} \lelr Q_{x_2}$. At first glance, this looks like a rather strong assumption coming out of thin air, but it is familiar from mathematical statistics or discriminant analyses and has interesting properties. For instance, $Q_{x_1} \lelr Q_{x_2}$ if and only if $Q_{x_1}(\cdot \,|\, B) \lest Q_{x_2}(\cdot \,|\, B)$ for any real interval $B$ such that $Q_{x_1}(B), Q_{x_2}(B) > 0$, where $Q_{x_j}(A \,|\, B) := Q_{x_j}(A\cap B)/Q_{x_j}(B)$. Furthermore, likelihood ratio ordering is a frequent assumption or implication of models in mathematical finance, see \cite{Beare_2015}, \cite{Jewitt_1991}. The notion of likelihood ratio order is reviewed thoroughly in \cite{Duembgen_Moesching_2022}, showing that it defines a partial order on the set of all probability measures on the real line which is preserved under weak convergence. That material generalizes definitions and results in \cite{Shaked_2007}.

Thus far, estimation of distributions under a likelihood ratio order constraint was mainly limited to settings with two or finitely many samples and populations. First, \cite{Dykstra_1995} estimated the parameters of two multinomial distributions that are likelihood ratio ordered via a restricted maximum likelihood approach. After reparametrization, they found that the maximization problem at hand had reduced to a specific bioassay problem treated by \cite{Robertson_1988} and which makes use of the theory of isotonic regression. It is then suggested that their approach generalizes well to any two distributions that are absolutely continuous with respect to some dominating measure. Later, \cite{Carolan_2005} focused on testing procedures for the equality of two distributions $Q_1$ and $Q_2$ versus the alternative hypothesis that $Q_1 \lelr Q_2$, in the specific case where the cumulative distribution functions $G_i$ of $Q_i$, $i=1,2$, are continuous. To this end, they made use of the equivalence between likelihood ratio order and the convexity of the ordinal dominance curve $\alpha \mapsto G_2\bigl(G_1^{-1}(\alpha)\bigr)$, $\alpha\in[0,1]$, which holds in case of $G_2$ being absolutely continuous with respect to $G_1$. The convexity of the ordinal dominance curve was also exploited by \cite{Westling_2019} to provide nonparametric maximum likelihood estimators of $G_1$ and $G_2$ under likelihood ratio order for discrete, continuous, as well as mixed continuous-discrete distributions using the greatest convex minorant of the empirical ordinal dominance curve. However, this method still necessitates the restrictive assumption that $G_2$ is absolutely continuous with respect to $G_1$. Other attempts at estimating two likelihood ratio ordered distributions include \cite{Yu_2017} who treat the estimation problem with a maximum smoothed likelihood approach, requiring the choice of a kernel and bandwidth parameters, and \cite{Hu_etal_2022} who suppose absolutely continuous distributions and model the logarithm of the ratio of densities as a linear combination of Bernstein polynomials.

To the best of our knowledge, only \cite{Dardanoni_Forcina_1998} considered the problem of estimating an arbitrary fixed number $\ell\ge 2$ of likelihood ratio ordered distributions $Q_1,Q_2,\ldots,Q_\ell$, all of them sharing the same finite support. They showed that the constrained maximum likelihood problem may be reparametrized to obtain a convex optimization problem with linear inequality constraints, and they propose to solve the latter via a constrained version of the Fisher scoring algorithm. At each step of their procedure, it is necessary to solve a quadratic programming problem.

Within the setting of distributional regression, we follow an empirical likelihood approach \citep{Owen_1988,Owen_2001} to estimate the family $(Q_x)_{x\in\mathfrak{X}}$ for arbitrary real sets $\mathfrak{X}$.  After a reparametrization similar to that of \cite{Dardanoni_Forcina_1998}, we show that the problem of maximizing the (empirical) likelihood under the likelihood ratio order constraint yields again a finite-dimensional convex optimization problem with linear inequality constraints. We did experiments with active set algorithms in the spirit of \cite{Duembgen_etal_2021} which are similar to the algorithms of \cite{Dardanoni_Forcina_1998}. But, as explained later, the computational burden may become too heavy for large sample sizes~$n$. Alternatively, we devise an algorithm which adapts and extends ideas from \cite{Jongbloed_1998} and \cite{Duembgen_2006} for the present setting. It makes use of a quasi-Newton approach, and new search directions are obtained via multiple isotonic weighted least squares regression.

There is an interesting aspect of the present estimation problem. If we assume that the observations $(X_i,Y_i)$ are independent copies of a generic random pair $(X,Y)$, the new estimation method may also be interpreted as an empirical likelihood estimator of the joint distribution of $(X,Y)$, hypothesizing that the latter is bivariate totally positive of order two (TP2). That is, for arbitrary intervals $A_1, A_2$ and $B_1, B_2$ such that $A_1 < A_2$ and $B_1 < B_2$ element-wise,
\[
	\Pr(X\in A_2, Y\in B_1) \Pr(X\in A_1, Y\in B_2) \
	\le \ \Pr(X\in A_1, Y\in B_1) \Pr(X\in A_2, Y\in B_2) .
\]
If the joint distribution of $(X,Y)$ has a density $h$ with respect to Lebesgue measure on $\R\times\R$, or if it is discrete with probability mass function $h$, then TP2 is equivalent to requiring that
\[
	h(x_1,y_2) h(y_1,x_2) \ \le \ h(x_1,y_1) h(x_2, y_2)
	\quad\text{whenever} \ \ x_1 < x_2, \ y_1 < y_2 ,
\]
and this is just a special case of multivariate total positivity of order two \citep{Karlin_1968}. For further equivalences and results in dimension two, see \cite{Duembgen_Moesching_2022}. Interestingly, this TP2 constraint is symmetric in $X$ and $Y$, and our algorithm exploits this symmetry. A different, more restrictive approach to the estimation of a TP2 distribution is proposed by \cite{Hutter_etal_2022}. They assume that the distribution of $(X,Y)$ has a smooth density with respect to Lebesgue measure on a given rectangle and devise a sieve maximum likelihood estimator.

The rest of the article is structured as follows. Section~\ref{Sec:EL} explains why empirical likelihood estimation of a family of likelihood ratio ordered distributions is essentially equivalent to the estimation of a discrete bivariate TP2 distribution. In Section~\ref{Sec:Estimation} we present an algorithm to estimate a bivariate TP2 distribution. In Section~\ref{Sec:Simulation.study}, a simulation study illustrates the benefits of the new estimation paradigm compared to the usual stochastic order constraint. Proofs and technical details are deferred to the appendix.

\section{Two versions of empirical likelihood modelling}
\label{Sec:EL}
%
With our observations $(X_i,Y_i)\in\mathfrak{X}\times \R$, $1 \le i \le n$, let
\[
	\{X_1, X_2, \ldots, X_n\} \ = \ \{x_1, \ldots, x_\ell\}
	\quad\text{and}\quad
	\{Y_1, Y_2, \ldots, Y_n\} \ = \ \{y_1, \ldots, y_m\},
\]
with $x_1 < \cdots < x_\ell$ and $y_1 < \cdots < y_m$. For an index pair $(j,k)$ with $1 \le j \le \ell$ and $1 \le k \le m$, let
\[
	w_{jk} \ := \ \# \bigl\{ i : (X_i,Y_i) = (x_j,y_k) \bigr\} .
\]
That means, the empirical distribution $\hat{R}_{\rm emp}$ of the observations $(X_i,Y_i)$ can be written as $\hat{R}_{\rm emp} = n^{-1} \sum_{j=1}^\ell \sum_{k=1}^m w_{jk}^{} \delta_{(x_j,y_k)}^{}$.

\subsection{Estimating the conditional distributions \texorpdfstring{$Q_x$}{Qx}}
%
To estimate $(Q_x)_{x \in \mathfrak{X}}$ under likelihood ratio ordering, we first estimate $(Q_{x_j})_{1 \le j \le \ell}$. If that results in $(\hat{Q}_{x_j})_{1 \le j \le \ell}$, we may define
\[
	\hat{Q}_x \ := \ \begin{cases}
		\hat{Q}_{x_1}
			& \text{if} \ x < x_1 , \\
		(1 - \lambda) \hat{Q}_{x_j} + \lambda \hat{Q}_{x_{j+1}}
			& \text{if} \ x = (1 - \lambda) x_j + \lambda x_{j+1}, \
				1 \le j < \ell, \ 0 < \lambda < 1 , \\
		\hat{Q}_{x_\ell}
			& \text{if} \ x > x_\ell .
	\end{cases}
\]
This piecewise linear extension preserves isotonicity with respect to $\lelr$, see Lemma~\ref{Lem:Lin.interpolation.lr}.

To estimate $Q_{x_1}, \ldots, Q_{x_\ell}$, we restrict our attention to distributions with support $\{y_1,\ldots,y_m\}$. That means, we assume temporarily that for $1 \le j \le \ell$,
\[
	Q_{x_j} \ = \ \sum_{k=1}^m q_{jk}^{} \delta_{y_k}^{}
\]
with weights $q_{j1}, \ldots, q_{jm} \ge 0$ summing to one. The empirical log-likelihood for the corresponding matrix $\bs{q} = (q_{jk})_{j,k} \in [0,1]^{\ell \times m}$ equals
\begin{equation}
\label{eq:ELLraw}
	L_{\rm raw}(\bs{q}) \ := \ \sum_{j=1}^\ell \sum_{k=1}^m w_{jk}^{} \log q_{jk}^{} .
\end{equation}
Then the goal is to maximize this log-likelihood over all matrices $\bs{q} \in [0,1]^{\ell\times m}$ such that
\begin{align}
\label{eq:EL1}
	\sum_{k=1}^m q_{jk}^{} \
	&= \ 1
		&&\text{for} \ 1 \le j \le \ell , \\
\label{ineq:EL}
	q_{j_1k_2}^{} q_{j_2k_1}^{} \
	&\le \ q_{j_1k_1}^{} q_{j_2k_2}^{}
		&&\text{for} \ 1 \le j_1 < j_2 \le \ell
		\ \text{and} \ 1 \le k_1 < k_2 \le m .
\end{align}
The latter constraint is equivalent to saying that $Q_{x_j}$ is isotonic in $j \in \{1,\ldots,\ell\}$ with respect to $\lelr$.

\subsection{Estimating the distribution of \texorpdfstring{$(X,Y)$}{(X,Y)}}
%
Suppose that the observations $(X_i,Y_i)$ are independent copies of a random pair $(X,Y)$ with unknown TP2 distribution $R$ on $\R\times\R$. An empirical likelihood approach to estimating $R$ is to restrict one's attention to distributions
\[
	R \ = \ \sum_{j=1}^\ell \sum_{k=1}^m h_{jk}^{} \delta_{(x_j,y_k)}^{}
\]
with $\ell m$ weights $h_{jk} \ge 0$ summing to one. The empirical log-likelihood of the corresponding matrix $\bs{h} = (h_{jk})_{j,k}$ equals $L_{\rm raw}(\bs{h})$ with the function $L_{\rm raw}$ defined in \eqref{eq:ELLraw}. But now the goal is to maximize $L_{\rm raw}(\bs{h})$ over all matrices $\bs{h} \in [0,1]^{\ell\times m}$ satisfying the constraints
\begin{equation}
\label{eq:EL2}
	\sum_{j=1}^\ell \sum_{k=1}^m h_{jk}^{}
		\ = \ 1
\end{equation}
and \eqref{ineq:EL}. As mentioned in the introduction, requirement \eqref{ineq:EL} for $\bs{h}$ is equivalent to $R$ being TP2. One can get rid of the constraint \eqref{eq:EL2} via a Lagrange trick and maximize
\[
	L(\bs{h}) \ := \ L_{\rm raw}(\bs{h}) - n h_{++} + n
\]
over all $\bs{h}$ satisfying \eqref{ineq:EL}, where $h_{++} := \sum_j \sum_k h_{jk}$. Indeed, if $\bs{h}$ is a matrix in $[0,\infty)^{\ell\times m}$ such that $L_{({\rm raw})}(\bs{h}) > - \infty$, then $\tilde{\bs{h}} := (h_{jk}/h_{++})_{j,k}$ satisfies \eqref{ineq:EL} if and only if $\bs{h}$ does, and
\[
	L(\bs{h}) \ = \ L_{\rm raw}(\tilde{\bs{h}}) + n (\log h_{++} - h_{++} + 1)
	\ \le \ L_{\rm raw}(\tilde{\bs{h}}) \ = \ L(\tilde{\bs{h}})
\]
with equality if and only if $h_{++} = 1$, that is, $\bs{h} = \tilde{\bs{h}}$.

\subsection{Equivalence of the two estimation problems}
%
For any matrix $\bs{a} \in \R^{\ell\times m}$ define the row sums $a_{j+} := \sum_k a_{jk}$ and column sums $a_{+k} := \sum_j a_{jk}$. If $\bs{h}$ is an arbitrary matrix in $[0,\infty)^{\ell\times m}$ such that $L_{\rm raw}(\bs{h}) > - \infty$, and if we write
\[
	h_{jk}^{} \ = \ p_j^{} q_{jk}^{}
	\quad\text{with} \ p_j^{} := h_{j+}^{} \ \text{and} \ q_{jk}^{} := h_{jk}^{} / h_{j+}^{} ,
\]
then $\bs{h}$ satisfies \eqref{ineq:EL} if and only if $\bs{q}$ does. Furthermore, $\bs{q}$ satisfies \eqref{eq:EL1}, and elementary algebra shows that
\[
	L(\bs{h}) \ = \ L_{\rm raw}(\bs{q})
		+ \sum_{j=1}^\ell \bigl( w_{j+}^{} \log p_j^{} - n p_j^{} + w_{j+}^{} \bigr) .
\]
The unique maximizer $\bs{p} = (p_j)_j$ of $\sum_j (w_{j+} \log p_j - n p_j + w_{j+})$ is the vector $(w_{j+}/n)_j$, and this implies the following facts:
\begin{itemize}
\item If $\hat{\bs{h}}$ is a maximizer of $L(\bs{h})$ under the constraints \eqref{ineq:EL}, then $\hat{h}_{j+} = w_{j+}/n$ for all $j$, and $\hat{q}_{jk} := \hat{h}_{jk}/\hat{h}_{j+}$ defines a maximizer $\hat{\bs{q}}$ of $L_{\rm raw}(\bs{q})$ under the constraints \eqref{eq:EL1} and \eqref{ineq:EL}.
\item If $\hat{\bs{q}}$ is a maximizer of $L_{\rm raw}(\bs{q})$ under the constraints \eqref{eq:EL1} and \eqref{ineq:EL}, then $\hat{h}_{jk} := (w_{j+}/n) \hat{q}_{jk}$ defines a maximizer $\hat{\bs{h}}$ of $L(\bs{h})$ under the constraints \eqref{ineq:EL}.
\end{itemize}

As a final remark, note that the two estimation problems are monotone equivariant in the following sense: If $(X, Y)$ is replaced with $(\tilde{X}, \tilde{Y})=(\sigma(X), \tau(Y))$ with strictly isotonic functions $\sigma:\mathfrak{X}\to\R$ and $\tau:\R \to \R$, then $\mathcal{L}(\tilde{Y}|\tilde{X}=\sigma(x)) = \mathcal{L}(\tau(Y)|X = x)$ for $x\in\mathfrak{X}$. Furthermore, the constraints of likelihood ratio ordered conditional distributions or of a TP2 joint distribution remain valid under such transformations.

\subsection{Calibration of rows and columns}
%
The previous considerations motivate to find a maximizer $\hat{\bs{h}} \in [0,\infty)^{\ell\times m}$ of $L(\bs{h})$ under the constraint \eqref{ineq:EL}, even if the ultimate goal is to estimate the conditional distributions $Q_x$, $x \in \mathfrak{X}$. They also indicate two simple ways to improve a current candidate $\bs{h}$ for $\hat{\bs{h}}$. Let $\tilde{\bs{h}}$ be defined via
\[
	\tilde{h}_{jk}^{} \ := \ (w_{j+}^{}/n) h_{jk}^{}/h_{j+}^{} ,
\]
i.e.\ we rescale the rows of $\bs{h}$ such that the new row sums $\tilde{h}_{j+}$ coincide with the empirical weights $w_{j+}/n$. Then
\[
	L(\tilde{\bs{h}}) - L(\bs{h})
	\ = \ \sum_{j=1}^\ell
		\Bigl( w_{j+}^{} \log \Bigl( \frac{w_{j+}}{n h_{j+}} \Bigr)
			+ n h_{j+}^{} - w_{j+}^{} \Bigr)
	\ \ge \ 0
\]
with equality if and only if $\tilde{\bs{h}} = \bs{h}$. Similarly, one can improve $\bs{h}$ by rescaling its columns, i.e.\ replacing $\bs{h}$ with $\tilde{\bs{h}}$, where
\[
	\tilde{h}_{jk}^{} \ := \ (w_{+k}^{}/n) h_{jk}^{}/h_{+k}^{} .
\]

\section{Estimation}
\label{Sec:Estimation}
%
\subsection{Dimension reduction}
%
The minimization problem mentioned before involves a parameter $\bs{h} \in [0,\infty)^{\ell\times m}$ under $\binom{\ell}{2} \binom{m}{2}$ nonlinear inequality constraints. The parameter space and the number of constraints may be reduced as follows.

\begin{Lemma}
\label{Lem:Dim.reduction}
Let $\PP$ be the set of all index pairs $(j,k)$ such that there exist indices $1 \le j_1 \le j \le j_2 \le \ell$ and $1 \le k_1 \le k \le k_2 \le m$ with $w_{j_1k_2}, w_{j_2k_1} > 0$.\\[1ex]
\textbf{(a)} \ If $\bs{h} \in [0,\infty)^{\ell\times m}$ satisfies \eqref{ineq:EL} and $L(\bs{h}) > - \infty$, then $h_{jk} > 0$ for all $(j,k) \in \PP$.\\[1ex]
\textbf{(b)} \ If such a matrix $\bs{h}$ is replaced with $\tilde{\bs{h}} := \bigl( 1_{[(j,k) \in \PP]} h_{jk} \bigr)_{j,k}$, then $\tilde{\bs{h}}$ satisfies \eqref{ineq:EL}, too, and $L(\tilde{\bs{h}}) \ge L(\bs{h})$ with equality if and only if $\tilde{\bs{h}} = \bs{h}$.\\[1ex]
\textbf{(c)} \ If $\bs{h} \in [0,\infty)^{\ell\times m}$ such that $\{(j,k)\colon h_{jk} > 0\} = \PP$, then constraint~\eqref{ineq:EL} is equivalent to
\begin{equation}
\label{ineq:EL2}
	h_{j-1,k}^{} h_{j,k-1} \
	\le \ h_{j-1,k-1}^{} h_{j,k}^{}
	\quad\text{for} \ 1 < j \le \ell
		\ \text{and} \ 1 < k \le m .
\end{equation}
\end{Lemma}

All in all, we may restrict our attention to parameters $\bs{h} \in (0,\infty)^{\PP}$ satisfying \eqref{ineq:EL2}, where $h_{jk} := 0$ for $(j,k) \not\in \PP$. Note that \eqref{ineq:EL2} involves only $(\ell-1)(m-1)$ inequalities, and the inequality for one particular index pair $(j,k)$ is nontrivial only if the two pairs $(j-1,k),(j,k-1)$ belong to $\PP$.

The set $\PP$ consists of all pairs $(j,k)$ such that the support of the empirical distribution $\hat{R}_{\rm emp}$ contains a point $(x_{j_1},y_{k_2})$ ``northwest'' and a point $(x_{j_2},y_{k_1})$ ``southeast'' of $(x_j,y_k)$. If $\PP$ contains two pairs $(j_2,k_1), (j_1,k_2)$ with $j_1 < j_2$ and $k_1 < k_2$, then it contains the whole set $\{j_1,\ldots,j_2\} \times \{k_1,\ldots,k_2\}$. Figure~\ref{Fig:Design} illustrates the definition of $\PP$. It also illustrates two alternative codings of $\PP$: An index pair $(j,k)$ belongs to $\PP$ if and only if $m_j \le k \le M_j$, where
\begin{align*}
	m_j \ &:= \ \min \bigl\{ k : w_{j'k} > 0 \ \text{for some} \ j' \ge j \bigr\} , \\
	M_j \ &:= \ \max \bigl\{ k : w_{j'k} > 0 \ \text{for some} \ j' \le j \bigr\} .
\end{align*}
Note that $m_j \le M_j$ for all $j$, $1 = m_1 \le \cdots \le m_\ell$, and $M_1 \le \cdots \le M_\ell = m$. Analogously, a pair $(j,k)$ belongs to $\PP$ if and only if $\ell_k \le j \le L_k$, where
\begin{align*}
	\ell_k \ &:= \ \min \bigl\{ j : w_{jk'} > 0 \ \text{for some} \ k' \ge k \bigr\} , \\
	L_k \ &:= \ \max \bigl\{ j : w_{jk'} > 0 \ \text{for some} \ k' \le k \bigr\} .
\end{align*}
Here $\ell_k \le L_k$ for all $k$, $1 = \ell_1 \le \cdots \le \ell_M$, and $L_1 \le \cdots \le L_m = \ell$.

\begin{figure}
\centering
\includegraphics[width=0.8\textwidth]{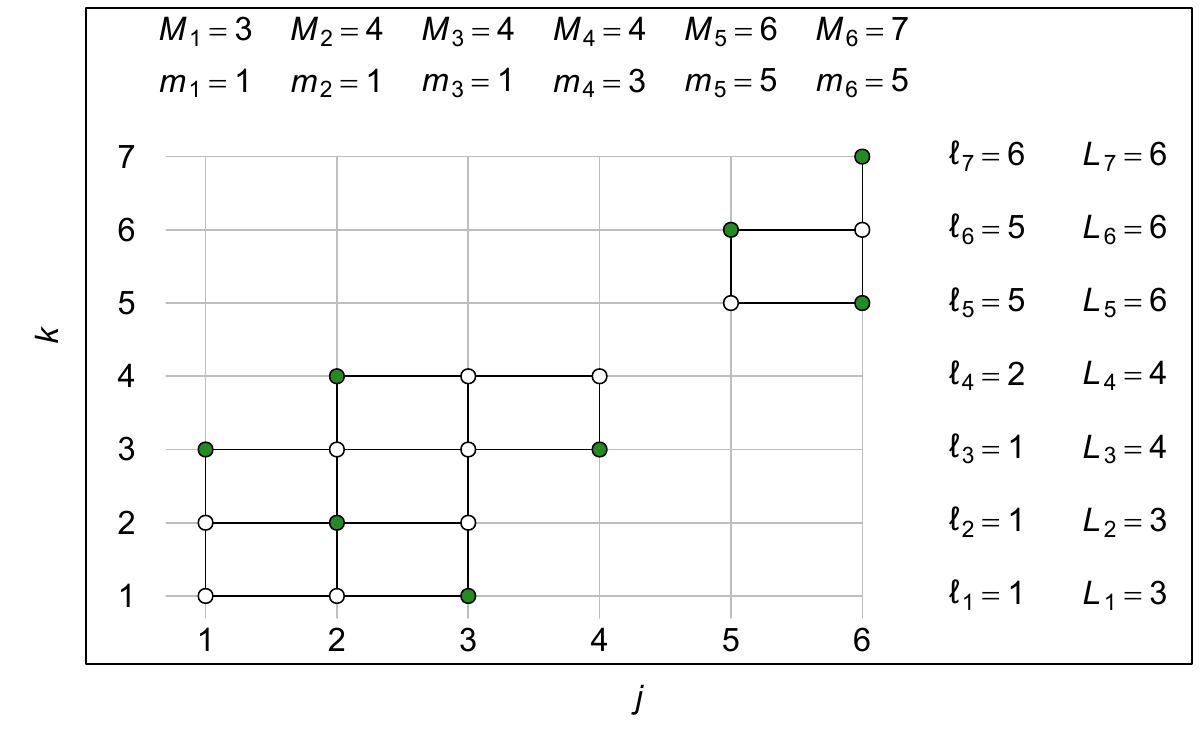}
\caption{In this specific example, $n \ge 8$ raw observations yielded $\ell = 6$ different values $x_j$ and $m = 7$ different values $y_k$. The green dots represent those $(j,k)$ with $w_{jk} > 0$. The green dots and black circles represent the set $\PP$.}
\label{Fig:Design}
\end{figure}

Note that by definition, for any index pair $(j,k)$,
\begin{align}
	k \le M_j \quad \text{if and only if} \quad j \ge l_k, \label{eq:northwest}\\
	k \ge m_j \quad \text{if and only if} \quad j \le L_k. \label{eq:southeast}
\end{align}

\subsection{Reparametrization and reformulation}
%
If we replace a parameter $\bs{h} \in (0,\infty)^{\PP}$ with its component-wise logarithm $\thb \in \R^{\PP}$, then property \eqref{ineq:EL2} is equivalent to
\begin{equation}
\label{ineq:ELL2}
	\theta_{j-1,k-1}^{} + \theta_{j,k}^{} - \theta_{j-1,k}^{} - \theta_{j,k-1}^{} \ \ge \ 0
	\quad \text{whenever} \ (j-1,k), (j,k-1) \in \PP .
\end{equation}
The set of all $\thb \in \R^{\PP}$ satisfying \eqref{ineq:ELL2} is a closed convex cone and is denoted by $\Theta$.

Now our goal is to minimize
\begin{equation}
\label{eq:fELL}
	f(\thb) \ := \ \sum_{(j,k) \in \PP}
		\bigl( - w_{jk}^{} \theta_{jk}^{} + n \exp(\theta_{jk}^{}) \bigr)
\end{equation}
over all $\thb \in \Theta$.

\begin{Theorem}
\label{Thm:Existence.uniqueness}
There exists a unique minimizer $\hat{\thb}$ of $f(\thb)$ over all $\thb \in \Theta$.
\end{Theorem}

Uniqueness follows directly from $f$ being strictly convex, but existence is less obvious, unless $w_{jk} > 0$ for all $(j,k)$. With $\hat{\thb}$ at hand, the corresponding solution $\hat{\bs{h}} \in [0,\infty)^{\ell\times m}$ of the original problem is given by
\[
	\hat{h}_{jk} \ = \ \begin{cases}
		\exp(\hat{\theta}_{jk}) & \text{if} \ (j,k) \in \PP , \\
		0 & \text{else} .
	\end{cases}
\]
In the proof of Theorem~\ref{Thm:Existence.uniqueness} and from now on, we view $\R^{\PP}$ as a Euclidean space with inner product $\langle\x,\bs{y}\rangle := \sum_{(j,k) \in \PP} x_{jk}^{} y_{jk}^{}$ and the corresponding norm $\|\x\| := \langle\x,\x\rangle^{1/2}$. For a differentiable function $f : \R^{\PP} \to \R$, its gradient is defined as $\nabla f(\x) := \bigl( \partial f(\x) / \partial x_{jk}^{} \bigr)_{(j,k) \in \PP}$.

Let us explain briefly why traditional optimization algorithms may become infeasible for large sample sizes $n$. Depending on the input data, the set $\PP$ may contain more than $cn^2$ parameters, and the constraint \eqref{ineq:ELL2} may involve at least $cn^2$ linear inequalities, where $c > 0$ is some generic constant. Even if we restrict our attention to parameters $\thb \in \Theta$ such that a given subset of the inequalities in \eqref{ineq:ELL2} are equalities, they span a linear space of dimension at least $\max(\ell,m)$, because all parameters $\theta_{jm_j}$ and $\theta_{\ell_kk}$ are unconstrained, and $\max(\ell,m)$ may be at least $cn$. Just determining a gradient and Hessian matrix of the target function $f$ within this linear subspace would then require at least $cn^4$ steps. Consequently, traditional minimization algorithms involving exact Newton steps may be computationally infeasible. Alternatively, we propose an iterative algorithm with quasi Newton steps each of which has running time $O(n^2)$, and the required memory is of this order, too.

\subsection{Finding a new proposal}
%
\paragraph{Version~1.}
To determine whether a given parameter $\thb \in \R^{\PP}$ is already optimal and, if not, to obtain a better one, we reparametrize the problem a second time. Let $\thbtil = T(\thb) \in \R^{\PP}$ be given by
\[
	\tilde{\theta}_{jk} \ = \ \begin{cases}
		\theta_{jm_j} & \text{if} \ k = m_j , \\
		\theta_{jk} - \theta_{j,k-1} & \text{if} \ m_j < k \le M_j .
	\end{cases}
\]
Then $\thb = T^{-1}(\thbtil) = \bigl( \sum_{k'=m_j}^k \tilde{\theta}_{jk'} \bigr)_{j,k}$, and $f(\thb)$ is equal to
\begin{align*}
	\tilde{f}(\thbtil) \
	:= \ &\sum_{j=1}^\ell \sum_{k=m_j}^{M_j}
		\Bigl( - w_{jk} \sum_{k'=m_j}^k \tilde{\theta}_{jk'}
			+ n \exp \Bigl( \sum_{k'=m_j}^k \tilde{\theta}_{jk'} \Bigr) \Bigr) \\
	= \ &\sum_{j=1}^\ell \sum_{k=m_j}^{M_j}
		\Bigl( - \underline{w}_{jk} \tilde{\theta}_{jk}
			+ n \exp \Bigl( \sum_{k'=m_j}^k \tilde{\theta}_{jk'} \Bigr) \Bigr)
		\quad\text{with} \ \underline{w}_{jk} := \sum_{k'=k}^{M_j} w_{jk'} .
\end{align*}
More importantly, we may represent $\PP$ as
\begin{align*}
	\PP \ &= \  \bigl\{ (j,m_j) : 1 \le j \le \ell \bigr\}
		\cup \bigl\{ (j,k) : 1 \le j \le \ell, m_j < k \le M_j \bigr\} \\
		\ &= \ \bigl\{ (j,m_j) : 1 \le j \le \ell \bigr\}
		\cup \bigcup_{k=2}^m \bigl\{ (j,k) : \ell_k \le j \le L_{k-1} \bigr\},
\end{align*}
where the latter equation follows from \eqref{eq:northwest} and \eqref{eq:southeast}. Now the constraints \eqref{ineq:ELL2} read
\begin{equation}
\label{ineq:ELL.tilde}
	\bigl( \tilde{\theta}_{jk} \bigr)_{j=\ell_k}^{L_{k-1}} \in \Rup^{L_{k-1} - \ell_k + 1}
	\quad \text{whenever} \ 2 \le k \le m \ \text{and} \
		L_{k-1} - \ell_k + 1 \ge 2 .
\end{equation}
Here $\Rup^d := \{ \bs{x} \in \R^d : x_1 \le \cdots \le x_d\}$. The set of $\thbtil \in \R^{\PP}$ satisfying \eqref{ineq:ELL.tilde} is denoted by $\tilde{\Theta}$.

For given $\thb$ and $\thbtil = T(\thb)$, we approximate $\tilde{f}(\xtil)$ by the quadratic function
\begin{align*}
	\xtil \ \mapsto \
	&\tilde{f}(\thbtil)
		+ \bigl\langle \nabla \tilde{f}(\thbtil), \xtil - \thbtil \bigr\rangle
		+ 2^{-1} \sum_{(j,k) \in \PP}
			\frac{\partial^2 \tilde{f}}{\partial \tilde{\theta}_{jk}^2}(\thbtil)
			(\tilde{x}_{jk} - \tilde{\theta}_{jk})^2 \\
	&= \ \mathrm{const}(\thb)
		+ 2^{-1} \sum_{(j,k) \in \PP} \tilde{v}_{jk}(\thb)
			(\tilde{x}_{jk} - \tilde{\gamma}_{jk}(\thb))^2 \\
	&= \ \mathrm{const}(\thb)
		+ 2^{-1} \sum_{j=1}^\ell \tilde{v}_{jm_j}(\thb)
			(\tilde{x}_{jm_j} - \tilde{\gamma}_{jm_j}(\thb))^2 \\
	&\qquad\qquad
		+ \ 2^{-1} \sum_{k=2}^m \sum_{\ell_k \le j \le L_{k-1}} \tilde{v}_{jk}(\thb)
			(\tilde{x}_{jk} - \tilde{\gamma}_{jk}(\thb))^2
\end{align*}
with
\begin{align*}
	\tilde{v}_{jk}(\thb) \
	&:= \ \frac{\partial^2 \tilde{f}}{\partial \tilde{\theta}_{jk}^2}(\thbtil)
		&&= \ n \sum_{k'=k}^{M_j} \exp(\theta_{jk'}) , \\
	\tilde{\gamma}_{jk}(\thb) \
	&:= \ \tilde{\theta}_{jk}
			- \tilde{v}_{jk}(\thb)^{-1}
				\frac{\partial \tilde{f}}{\partial \tilde{\theta}_{jk}}(\thbtil)
		&&= \ T_{jk}(\thb)
			+ \tilde{v}_{jk}(\thb)^{-1} \underline{w}_{jk} - 1 .
\end{align*}
This quadratic function of $\xtil$ is easily minimized over $\tilde{\Theta}$ via the pool-adjacent-violators algorithm, applied to the subtuple $(\tilde{x}_{jk})_{j=\ell_k}^{L_{k-1}}$ for each $k=2,\ldots,m$ separately. Then we obtain the proposal
\[
	\Psi^{\rm row}(\thb) \ := \ T^{-1}(\thbtil_*(\thb))
	\quad\text{with}\quad
	\thbtil_*(\thb) \ := \ \argmin_{\xtil \in \tilde{\Theta}}
		\sum_{(j,k) \in \PP} \tilde{v}_{jk}(\thb)
			(\tilde{x}_{jk} - \tilde{\gamma}_{jk}(\thb))^2 .
\]
Interestingly, if $\thb$ is row-wise calibrated in the sense that $n \sum_{k=m_j}^{M_j} \exp(\theta_{jk}) = w_{j+}$ for $1 \le j \le \ell$, then $\tilde{\gamma}_{jm_j}(\thb) = \tilde{\theta}_{jm_j}$ and thus $\Psi^{\rm row}_{jm_j}(\thb) = \theta_{jm_j}$ for $1 \le j \le \ell$.

\paragraph{Version~2.}
Instead of reparametrizing $\thb \in \Theta$ in terms of its values $\theta_{jm_j}$, $1 \le j \le \ell$, and its increments within rows, one could reparametrize it in terms of its values $\theta_{\ell_kk}$, $1 \le k \le m$, and its increments within columns, leading to a proposal $\Psi^{\rm col}(\thb)$. Here, $\Psi^{\rm col}_{\ell_kk}(\thb) = \theta_{\ell_kk}$ for $1 \le k \le m$, provided that $\thb$ is column-wise calibrated.

\subsection{Calibration}
%
In terms of the log-parametrization with $\thb \in \Theta$, the row-wise calibration mentioned earlier for $\bs h$ means to replace $\theta_{jk}$ with 
\[
	\theta_{jk} - \log \bigl( 
		\sum_{k'=m_j}^{M_j} \exp(\theta_{jk'}) 
	\bigr) + \log(w_{j+}/n).
\]
Analogously, replacing $\theta_{jk}$ with
\[
	\theta_{jk} - \log \bigl( 
		\sum_{j'=\ell_k}^{L_k} \exp(\theta_{j'k}) 
	\bigr) + \log(w_{+k}/n)
\]
leads to a column-wise calibrated parameter $\thb$. Iterating these calibrations alternatingly, leads to a parameter which is (approximately) calibrated, row-wise as well as column-wise.

\subsection{From new proposal to new parameter}
%
Both functions $\Psi = \Psi^{\rm row}, \Psi^{\rm col}$ have some useful properties summarized in the next lemma.

\begin{Lemma}
\label{Lem:Properties.Psi}
The function $\Psi$ is continuous on $\Theta$ with $\Psi(\hat{\thb}) = \hat{\thb}$. For $\thb \in \Theta \setminus \{\hat{\thb}\}$,
\[
	\delta(\thb) \ := \ \bigl\langle \nabla f(\thb), \thb - \Psi(\thb) \bigr\rangle
	\ > \ 0 ,
\]
\[
	f(\thb) - f(\hat{\thb}) \ \le \ \max \bigl( 2 \delta(\thb),
		\beta_1(\thb) \sqrt{\delta(\thb)} \|\thb - \hat{\thb}\| \bigr) ,
\]
and
\[
	\max_{t \in [0,1]} \,
		\Bigl( f(\thb) - f \bigl( (1 - t)\thb + t \Psi(\thb) \bigr) \Bigr)
	\ \ge \ \min \Bigl( 2^{-1} \delta(\thb),
		\frac{\delta(\thb)^2}{\beta_2(\thb) \|\thb - \Psi(\thb)\|^2} \Bigr)
\]
with continuous functions $\beta_1, \beta_2 : \Theta \to (0,\infty)$.
\end{Lemma}

In view of this lemma, we want to replace $\thb \ne \hat{\thb}$ with $(1 - t_*) \thb + t_* \Psi(\thb)$ for some suitable $t_* = t_*(\thb) \in [0,1]$ such that $f(\thb)$ really decreases. More specifically, with
\[
	\rho_{\thb}(t) \ := \ f(\thb) - f \bigl( (1 - t)\thb + t \Psi(\thb) \bigr) ,
\]
our goals are that for some constant $\kappa \in (0,1]$,
\[
	\rho_{\thb}(t_*) \ \ge \ \kappa \max_{t \in [0,1]} \rho_{\thb}(t) ,
\]
and in case of $\rho_{\thb}$ being (approximately) a quadratic function, $t_*$ should be (approximately) equal to $\argmax_{t \in [0,1]} \rho_{\thb}(t)$. For that, we proceed similarly as in \citet{Duembgen_2006}. We determine $t_o := 2^{-n_o}$ with $n_o$ the smallest integer such that $\rho_{\thb}(2^{-n_o}) \ge 0$. Then we define a Hermite interpolation of $\rho_{\bs\theta}$:
\[
	\tilde{\rho}_{\thb}(t)
	\ := \ \rho_{\thb}'(0) t - c_o t^2
	\quad\text{with}\quad
	c_o := t_o^{-1} \bigl( \rho_{\thb}'(0) - t_o^{-1}\rho_{\thb}(t_o) \bigr) \ > \ 0 .
\]
This new function is such that $\tilde{\rho}_{\thb}(t)=\rho_{\thb}(t)$ for $t = 0, t_o$, and $\tilde{\rho}_{\thb}'(0) = \rho_{\thb}'(0) > 0$. Since $\tilde{\rho}_{\thb}'(t) = \rho_{\thb}'(0) - 2 t c_o$, the maximizer of $\tilde{\rho}_{\thb}$ over $[0,t_o]$ is given by
\[
	t_* \ := \ \min \bigl( t_o, 2^{-1} \rho_{\thb}'(0)/c_o \bigr) .
\]
As shown in Lemma~1 of \citet{Duembgen_2006}, this choice of $t_*$ fulfils the requirements just stated, where $\kappa = 1/4$.

\subsection{Complete algorithms}
%
A possible starting point for the algorithm is given by $\thb^{(0)} := ( - \log(\#\PP) )_{(j,k) \in \PP}$, but any other parameter $\thb^{(0)} \in \Theta$ would work, too. Suppose we have determined already $\thb^{(0)}, \ldots, \thb^{(s)}$ such that $f(\thb^{(0)}) \ge \cdots \ge f(\thb^{(s)})$. Let $\Psi(\thb^{(s)})$ be a new proposal with $\Psi = \Psi^{\rm row}$ or $\Psi = \Psi^{\rm col}$, and let $\thb^{(s+1)} = (1 - t_*^{(s)}) \thb^{(s)} + t_*^{(s)} \Psi(\thb^{(s)})$ with $t_*^{(s)} = t_*(\thb^{(s)}) \in [0,1]$ as described before. No matter which proposal function $\Psi$ we are using in each step, the resulting sequence $(\thb^{(s)})_{s\ge 0}$ will always converge to $\hat{\thb}$.

\begin{Theorem}
\label{Thm:Convergence}
Let $(\thb^{(s)})_{s \ge 0}$ be the sequence just described. Then $\lim_{s \to \infty} \thb^{(s)} = \hat{\thb}$.
\end{Theorem}

Our numerical experiments showed that a particularly efficient refinement is as follows: Before computing a new proposal $\Psi(\thb^{(s)})$, one should calibrate $\thb^{(s)}$ in the sense that it is row-wise and column-wise calibrated. If $s$ is even, we compute $\Psi^{\rm row}(\thb^{(s)})$ to determine the next candidate $\thb^{(s+1)}$. If $s$ is odd, we compute $\Psi^{\rm col}(\thb^{(s)})$ to obtain $\thb^{(s+1)}$. The algorithm stops as soon as $\delta(\thb^{(s)}) = \bigl\langle \nabla f(\thb^{(s)}), \thb^{(s)} - \Psi(\thb^{(s)}) \bigr\rangle$ is smaller than a prescribed small threshold. Table~\ref{Tab:Algorithm} provides corresponding pseudo code.

\begin{table}
\[
	\begin{array}{|l|}
	\hline
	\thb \leftarrow \thb^{(0)} \\
	\delta \leftarrow \infty\\
	s \leftarrow 0\\
	\text{while} \ \delta \ge \delta_o \ \text{do}\\
	\strut\quad
			\thb \leftarrow \text{calibration of} \ \thb\\
	\strut\quad
		\text{if} \ s \ \text{is even, do}\\
	\strut\quad\quad
			(\bs{\psi},\delta) \leftarrow
			\bigl( \Psi^{\rm row}(\thb),
				\bigl\langle \nabla f(\thb),
					\thb - \Psi^{\rm row}(\thb) \bigr\rangle \bigr)\\
	\strut\quad
		\text{else}\\
	\strut\quad\quad
			(\bs{\psi},\delta) \leftarrow
			\bigl( \Psi^{\rm col}(\thb),
				\bigl\langle \nabla f(\thb),
					\thb - \Psi^{\rm col}(\thb) \bigr\rangle \bigr)\\
	\strut\quad
		\text{end if}\\
	\strut\quad
		\rho' \leftarrow \delta\\	
	\strut\quad
		\text{while} \ f(\bs{\psi}) > f(\thb) \ \text{do}\\
	\strut\quad\quad
			(\bs{\psi},\rho') \leftarrow
				\bigl( 2^{-1}(\thb + \bs{\psi}), 2^{-1} \rho' \bigr)\\
	\strut\quad
		\text{end while}\\
	\strut\quad
		t_* \leftarrow
			\min \bigl( 1, 2^{-1} \rho'/
				\bigl(\rho' - f(\thb) + f(\bs{\psi}) \bigr) \bigr)\\
	\strut\quad
		\thb \leftarrow
			(1 - t_*) \thb + t_* \bs{\psi}\\
	\strut\quad
		s \leftarrow s+1\\
	\text{end while}\\
	\hline
	\end{array}
\]
\caption{Pseudo code of our algorithm, returning an approximation $\thb$ of $\hat{\thb}$.}
\label{Tab:Algorithm}
\end{table}

\section{Simulation study}
\label{Sec:Simulation.study}
%
In this section, we compare estimation and prediction performances of the likelihood ratio order constrained estimator presented in this article with the estimator under usual stochastic order obtained via isotonic distributional regression. The latter estimator was mentioned briefly in the introduction. It is extensively discussed in \citet{Henzi_Ziegel_Gneiting_2021} and \citet{Moesching_2020}.

\subsection{A Gamma model}
%
We choose a parametric family of distributions from which we draw observations. We will then use these data to provide distribution estimates which we then compare with the truth. The specific model we have in mind is a family $(Q_x)_{x\in\mathfrak{X}}$ of Gamma distributions with densities
\[
	g_x(y)
	\ := \
	\frac{b(x)^{-a(x)}}{\Gamma\bigl(a(x)\bigr)} y^{a(x)-1} \exp\bigl(-y/b(x)\bigr),
\]
with respect to Lebesgue measure on $(0,\infty)$, with some shape function $a:\mathfrak{X}\to (0,\infty)$ and scale function $b:\mathfrak{X}\to (0,\infty)$. Then $Q_x$ is isotonic in $x \in \mathfrak{X}$ with respect to likelihood ratio ordering if and only if both functions $a$ and $b$ are isotonic. Recall that since the family is increasing in likelihood ratio order, it is also increasing with respect to the usual stochastic order.

\begin{figure}
\centering
\includegraphics[width=0.8\textwidth]{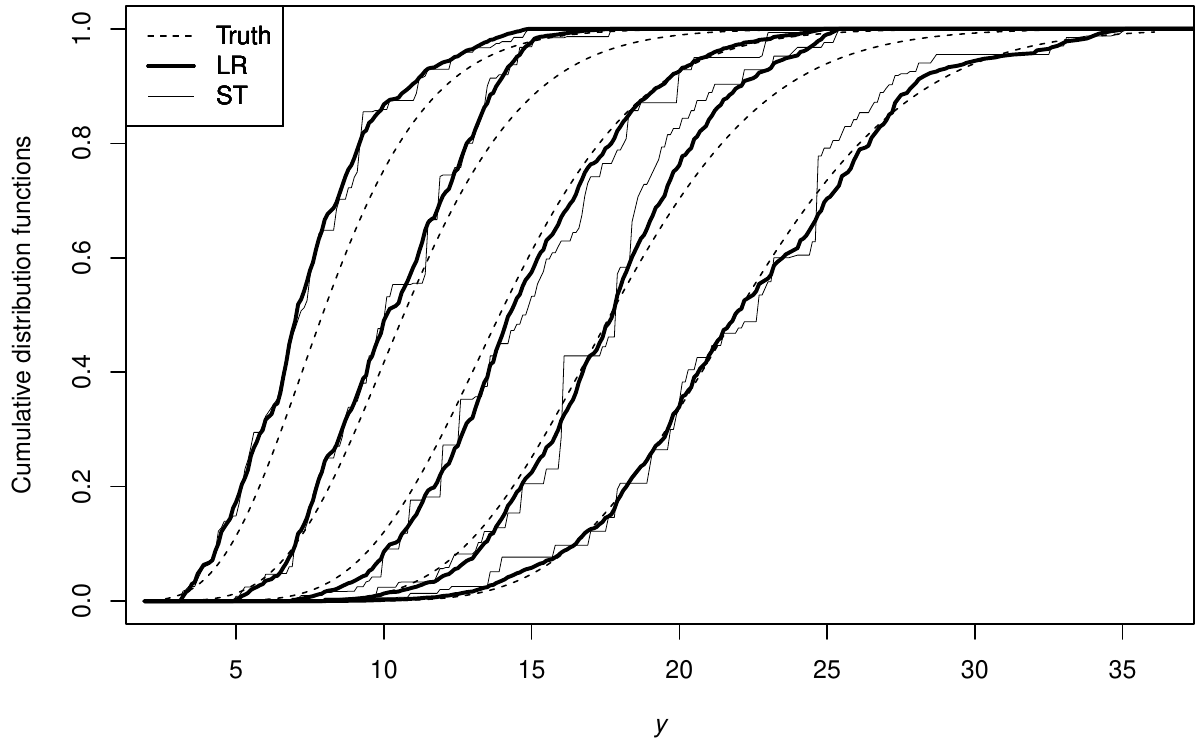}
\caption{The true conditional Gamma distribution function $G_x$, the estimate under likelihood ratio (LR) order constraint $\hat{G}_x$ and the estimated under usual stochastic (ST) order constraint $\check{G}_x$ are displayed from left to right for $x\in\{1.5,2,2.5,3,3.5\}$.}
\label{Fig:Gamma_CDFs}
\end{figure}

The specific shape and scale functions used for this study are
\[
	a(x) \ := \ 2+(x+1)^2
	\quad
	\text{and}
	\quad
	b(x) \ := \ 1 - \exp(-10x),
\]
defined for $x\in\mathfrak{X}:=[1,4]$. Figure~\ref{Fig:Gamma_CDFs} displays corresponding true conditional distribution functions for a selection of $x$'s.

\subsection{Sampling method}
\label{Sec:Sampling}
%
Let $\ell_o\in\{50,1000\}$ be a predefined number and let
\[
	\mathfrak{X}_o
	\ := \
	1 + \frac{3}{\ell_o} \cdot \{1,2,\ldots, \ell_o\}
	\ \subset \
	\mathfrak{X}.
\]
For a given sample size $n\in\N$, the sample $(X_1,Y_1),(X_2,Y_2),\ldots,(X_n,Y_n)$ is obtained as follows: Draw $X_1,X_2,\ldots,X_n$ uniformly from $\mathfrak{X}_o$ and sample independently each $Y_k$ from~$Q_{X_k}$. This yields unique covariates $x_1<\cdots<x_\ell$ as well as unique responses $y_1 < \cdots < y_m$, for some $1\le \ell,m \le n$.

For each such sample, we compute estimates of $(Q_{x_j})_{j=1}^\ell$ under likelihood ratio order and usual stochastic order constraints. Using linear interpolation, we complete both families of estimates with covariates originally in $\{x_j\}_{j=1}^\ell$ to families of estimates with covariates in the full set $\mathfrak{X}_o$, see Lemma~\ref{Lem:Lin.interpolation.lr}. We therefore obtain estimates $(\hat{Q}_{x})_{x\in\mathfrak{X}_o}$ and $(\check{Q}_{x})_{x\in\mathfrak{X}_o}$ under likelihood ratio order and usual stochastic order constraint, respectively. The corresponding families of cumulative distribution functions are written $(\hat{G}_x)_{x\in\mathfrak{X}_o}$ and $(\check{G}_x)_{x\in\mathfrak{X}_o}$, whereas the truth is denoted by $(G_x)_{x\in\mathfrak{X}_o}$. Although the performance of the empirical distribution is worse than those of the two order constrained estimators, it is still useful to study its behaviour, for instance to better understand boundary effects. The family of empirical cumulative distribution functions will be written $(\hat{\mathbb{G}}_x)_{x\in\mathfrak{X}_o}$.

\subsection{Single sample}
%
\begin{figure}
\centering
\includegraphics[width=0.8\textwidth]{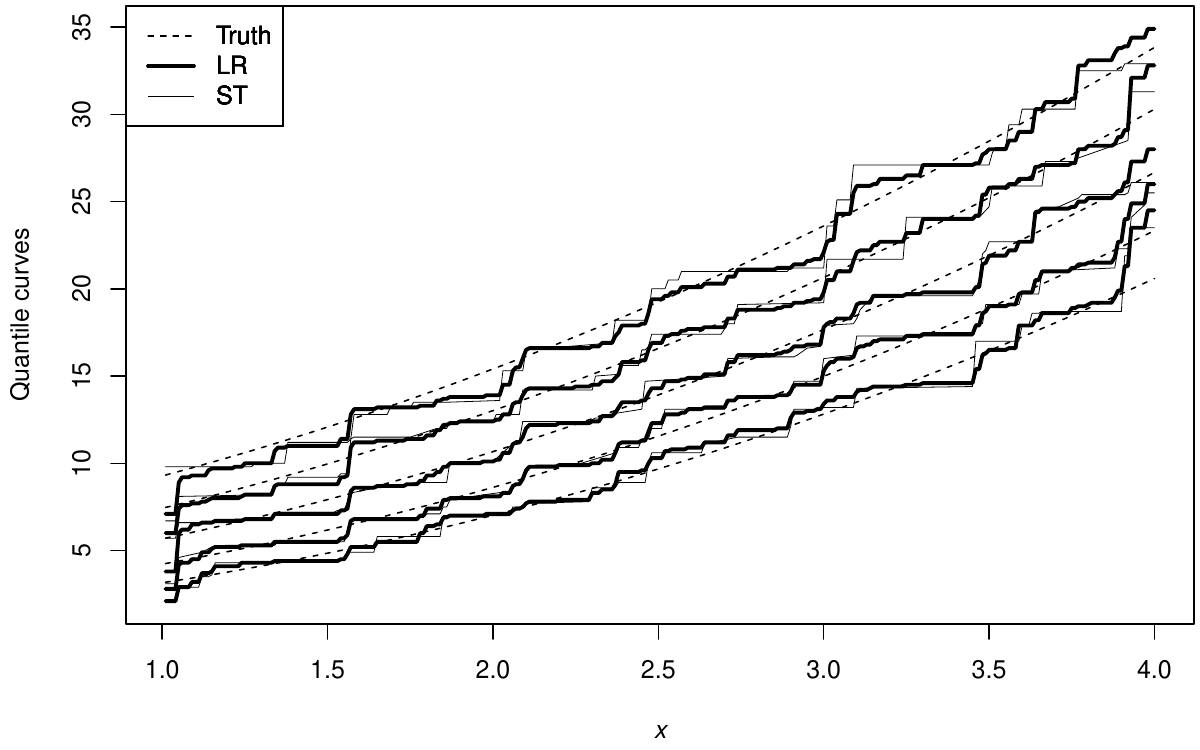}
\caption{Selection of $\beta$-quantile curves. Specifically, a taut-string \citep{Duembgen_Kovac_2009} is computed between the lower $\mathfrak{X}\ni x\mapsto \min\{y\in\R: \tilde{G}_x(y)\ge \beta\}$ and upper $\mathfrak{X}\ni x\mapsto \inf\{y\in\R: \tilde{G}_x(y)> \beta\}$ quantile curves for each $\tilde{G}\in\{G,\hat{G},\check{G}\}$ (corresponding respectively to `Truth', `LR' and `ST') and $\beta\in\{0.1,0.25,0.5,0.75,0.9\}$.}
\label{Fig:Gamma_Quantiles}
\end{figure}

Figure~\ref{Fig:Gamma_CDFs} provides a visual comparison of a selection of true conditional distribution functions with their corresponding estimates under order constraint for a single sample generated in the setting $\ell_o=1000$ and $n=1000$. It shows that the estimates under likelihood ratio order constraint are much smoother than those under usual stochastic order constraint. The former are in general also closer to the truth than the latter. This fact is in reality true on average, as demonstrated in the next paragraph. Smoothness and greater precision in estimation resulting from the likelihood ratio order is also apparent in Figure~\ref{Fig:Gamma_Quantiles}, which displays a selection of quantile curves for each $\tilde{G}\in\{G,\hat{G},\check{G}\}$.

\subsection{A simple score}
%
To assess the ability of each estimator to retrieve the truth, we produce Monte-Carlo estimates of the median of the score
\[
	R_x(\tilde{G}, G)
	\ := \
	\int \lvert \tilde{G}_x(y) - G_x(y) \rvert \, \mathrm{d}Q_x(y),
\]
for each estimator $\tilde{G}\in\{\hat{G},\check{G},\hat{\mathbb{G}}\}$ and for each $x\in\mathfrak{X}_o$. The above score may be decomposed as a sum of simple expressions involving the evaluation of $\tilde{G}_x$ and $G_x$ on the finite set of unique responses, see Section~\ref{Sec:RewritingScores}. We also compute Monte-Carlo quartiles of the relative change in score
\[
	100\cdot \frac{R_x(\hat{G}, G)-R_x(\check{G}, G)}{R_x(\check{G}, G)}.
\]

\begin{figure}
\centering
\includegraphics[width=0.8\textwidth]{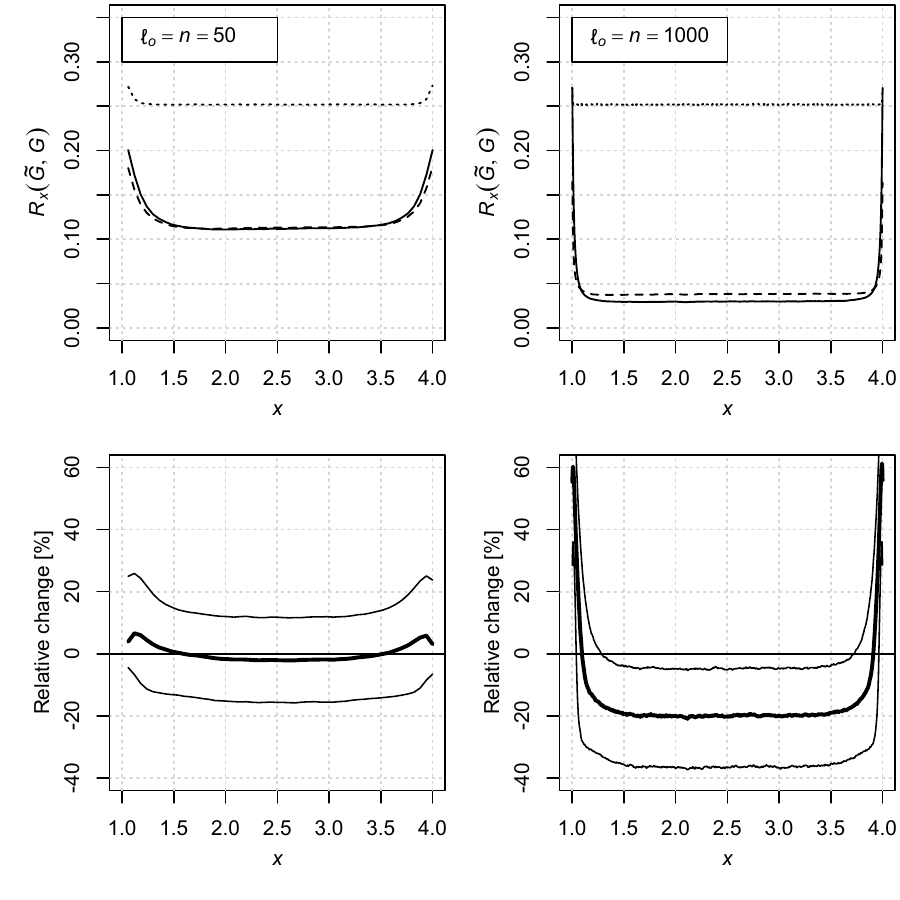}
\caption{Monte Carlo simulations to evaluate estimation performances with a simple score. First row: Simple scores with $\tilde{G}$ being either $\hat{G}$ (solid line), $\check{G}$ (dashed line) or $\hat{\mathbb{G}}$ (dotted line). Second row: Relative change of score when enforcing a likelihood ratio order constraint over the usual stochastic order constraint. The thicker line is the median variation, whereas the thin lines are the first and third quartiles. Negative values represent an improvement in score.}
\label{Fig:Gamma_SS}
\end{figure}

The results of the simulations are displayed in Figure~\ref{Fig:Gamma_SS}. A first observation is that the performance of all three estimators decreases towards the boundary points of $\mathfrak{X}$, and this effect is more pronounced for the two order constrained estimators. This is a known phenomenon from shape constrained inference. However, in the interior of $\mathfrak{X}$, taking the stochastic ordering into account pays off. The second row of plots in Figure~\ref{Fig:Gamma_SS} shows the relative change in score when estimating the family of distributions with a likelihood ratio order constraint instead of the usual stochastic order constraint. It is observed that the improvement in score becomes larger and occurs on a wider sub-interval of $\mathfrak{X}$ as $\ell_o$ and $n$ increase. Only towards the boundary, the usual stochastic order seems to have better performance.

\subsection{Theoretical predictive performances}
%
Using the same Gamma model, we evaluate predictive performances of both estimators using the continuous ranked probability score
\[
	\mathrm{CRPS}(\tilde{G}_x, y)
	\ := \
	\int \left(\tilde{G}_x(z) - 1_{[y\le z]}\right)^2 \, \mathrm{d}z.
\]
The CRPS is a sctrictly proper scoring rule which allows for comparisons of probabilistic forecasts, see \cite{Gneiting_2007} and \cite{Jordan_2019}. It can be seen as an extension of the mean absolute error for probabilistic forecasts. The CRPS is therefore interpreted in the same unit of measurement as the true distribution or data.

Because the true underlying distribution is known in the present simulation setting, the expected CRPS score is given by
\begin{align*}
	S_x(\tilde{G}, G)
	\ :=& \
	\int \mathrm{CRPS}(\tilde{G}_x, y) \, \mathrm{d}Q_x(y) \\
	\ =& \
	\sum_{k=0}^m \int_{[y_k,y_{k+1})}
		\bigl(
			\tilde{G}_x(y_k) - G_x(y)
		\bigr)^2 \,\mathrm{d}y
	+ \frac{b(x)}{B(1/2,a(x))},
\end{align*}
where $y_0:=0$, $y_{m+1}:=+\infty$ and $B(\cdot,\cdot)$ is the beta function. As shown in Section~\ref{Sec:RewritingScores}, the above sum of integrals may be rewritten as a sum of elementary expressions involving the evaluation of $\tilde{G}_{x}$ and $G_{x}$ on the finite set of unique responses, as well as two simple integrals which are computed via numerical integration. Consequently, we compute Monte-Carlo estimates of the median of each score $S_x(\tilde{G}, G)$, $\tilde{G}\in\{\hat{G},\check{G},\hat{\mathbb{G}}\}$, as well as estimates of quartiles of the relative change in score when choosing $\hat{G}$ over $\check{G}$.

\begin{figure}
\centering
\includegraphics[width=0.8\textwidth]{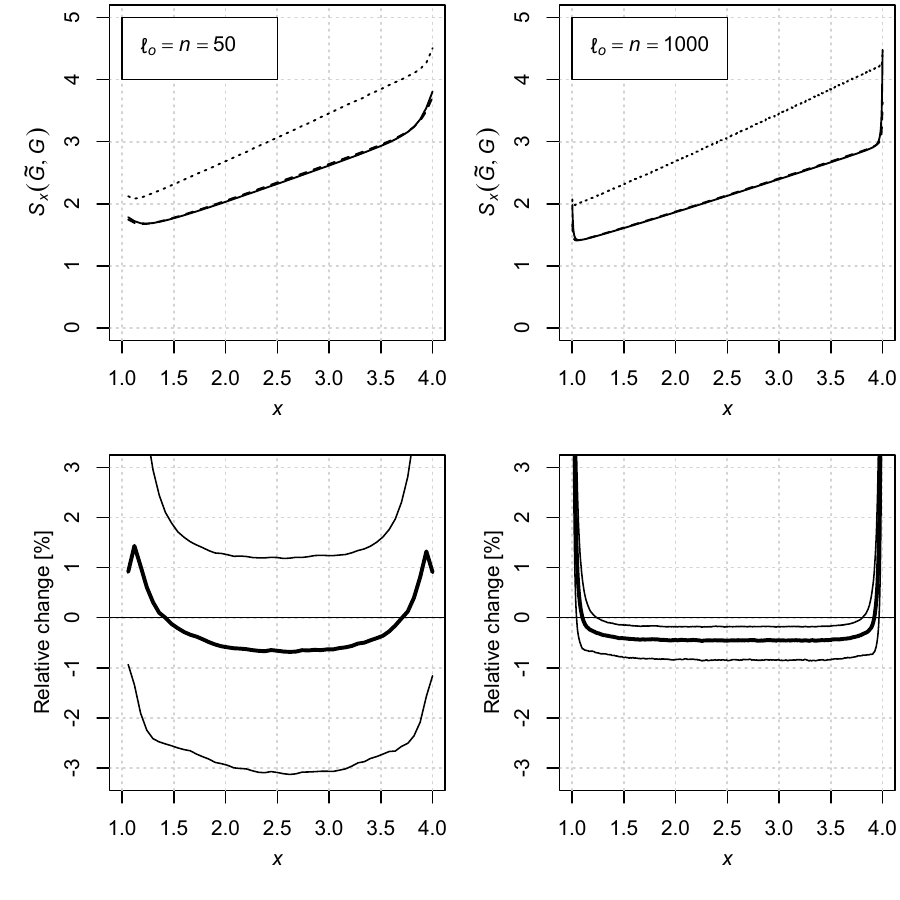}
\caption{Monte Carlo simulations to evaluate prediction performances using a CRPS-type score. First row: CRPS scores with $\tilde{G}$ being either $\hat{G}$ (solid line), $\check{G}$ (dashed line) or $\hat{\mathbb{G}}$ (doted line). Second row: Relative change of score when enforcing a likelihood ratio order constraint over the usual stochastic order constraint.}
\label{Fig:Gamma_CRPS}
\end{figure}

Figure~\ref{Fig:Gamma_CRPS} outlines the results of the simulations. Similar boundary effects as for the simple score are observed. On the interior of $\mathfrak{X}$, the usual stochastic order improves the naive empirical estimator, and the likelihood ratio order yields the best results. In terms of relative change in score, it appears that imposing a likelihood ratio order constraint to estimate the family of distributions yields an average score reduction of about $0.5\%$ in comparison with the usual stochastic order estimator for a sample of $n=50$. For $n=1000$, this improvement occurs on a wider subinterval of $\mathfrak{X}$ and more frequently, as shown by the third quartile curve. Note further that the expected CRPS increases on the interior of $\mathfrak{X}$. This is due to the fact that the CRPS has the same unit of measurement as the response variable. Since the scale of the response characterized by $b$ increases with $x$, then so does the corresponding score.

\subsection{Empirical predictive performances}
%
\begin{figure}
\centering
\includegraphics[width=0.8\textwidth]{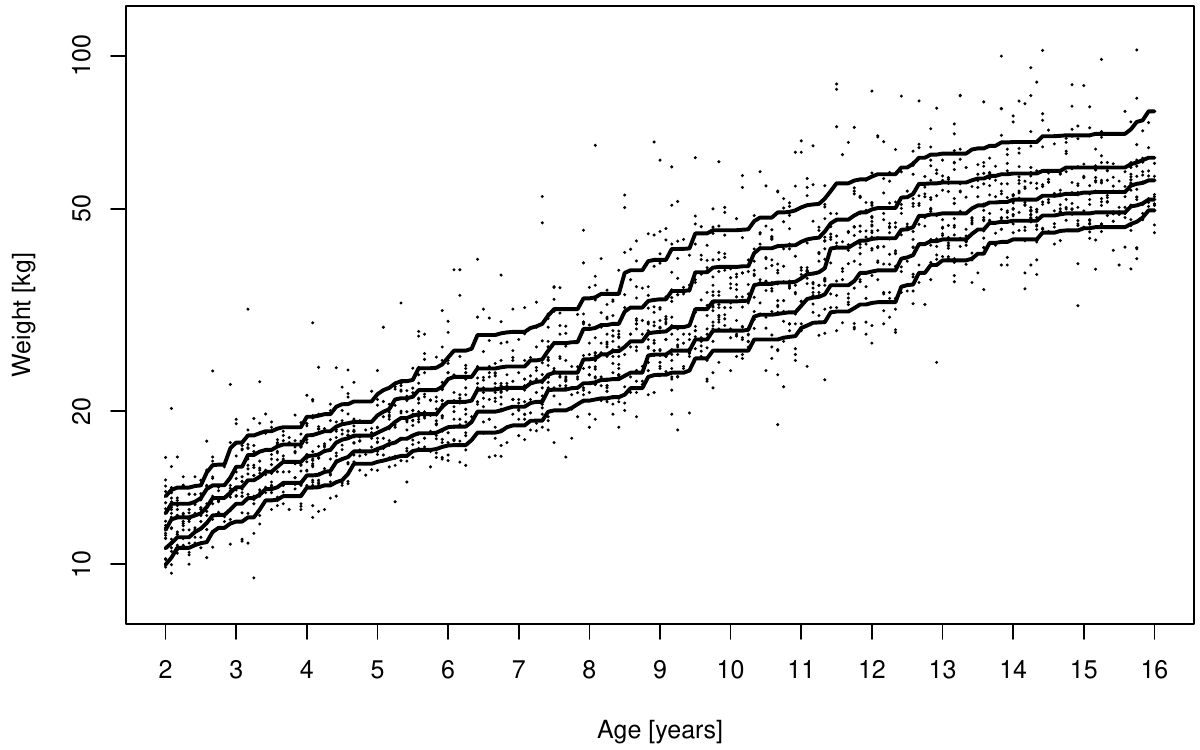}
\caption{Subsample of the weight for age data and $\beta$-quantile curves computed from that sample under likelihood ratio order constraint, $\beta\in\{0.1,0.25,0.5,0.75,0.9\}$. A logarithmic scale was used for the weight variable.}
\label{Fig:Growth_Sample}
\end{figure}

We use the weight for age dataset already studied in \cite{Moesching_2020}. It comprises the age and weight of $n=16\,432$ girls whose age in years lies within $\mathfrak{X}:=[2,16]$. A subsample of these data of size $2\,000$ is presented in Figure~\ref{Fig:Growth_Sample}, along with estimated quantile curves under likelihood ratio order using that subsample. The dataset was publicly released as part of the National Health and Nutrition Examination Survey conducted in the US between 1963 and 1991 (data available from \url{www.cdc.gov}) and was analyzed by \cite{CDC_2002} with parametric models to produce smooth quantile curves.

Although the likelihood ratio order constraint is harder to justify than the very natural stochastic order constraint, we are interested in the effect of a stronger regularization imposed by the former constraint.

The forecast evaluation is performed using a leave-$n_{\text{train}}$-out cross-validation scheme. More precisely, we choose random subsets $\DD_{\text{train}}$ of $n_{\text{train}}$ observations which we use to train our estimators. Using the rest of the $n_{\text{test}}:=n-n_{\text{train}}$ data pairs in $\DD_{\text{test}}$, we evaluate predictive performance by computing the sample median of $\hat{S}_x(\tilde{G},\DD_{\text{test}})$ for each estimator $\tilde{G}\in\{\hat{G},\check{G},\hat{\mathbb{G}}\}$ and each $x\in\mathfrak{X}_o$, where
\[
	\hat{S}_x(\tilde{G},\DD_{\text{test}})
	\ := \
	\frac{\sum_{(X,Y)\in\DD_{\text{test}}:X=x}
	\mathrm{CRPS}(\tilde{G}_x, Y)}{\#\{(X,Y)\in\DD_{\text{test}}:X=x\}}.
\]
Quartile estimates of the relative change in score are also computed.

\begin{figure}
\centering
\includegraphics[width=0.8\textwidth]{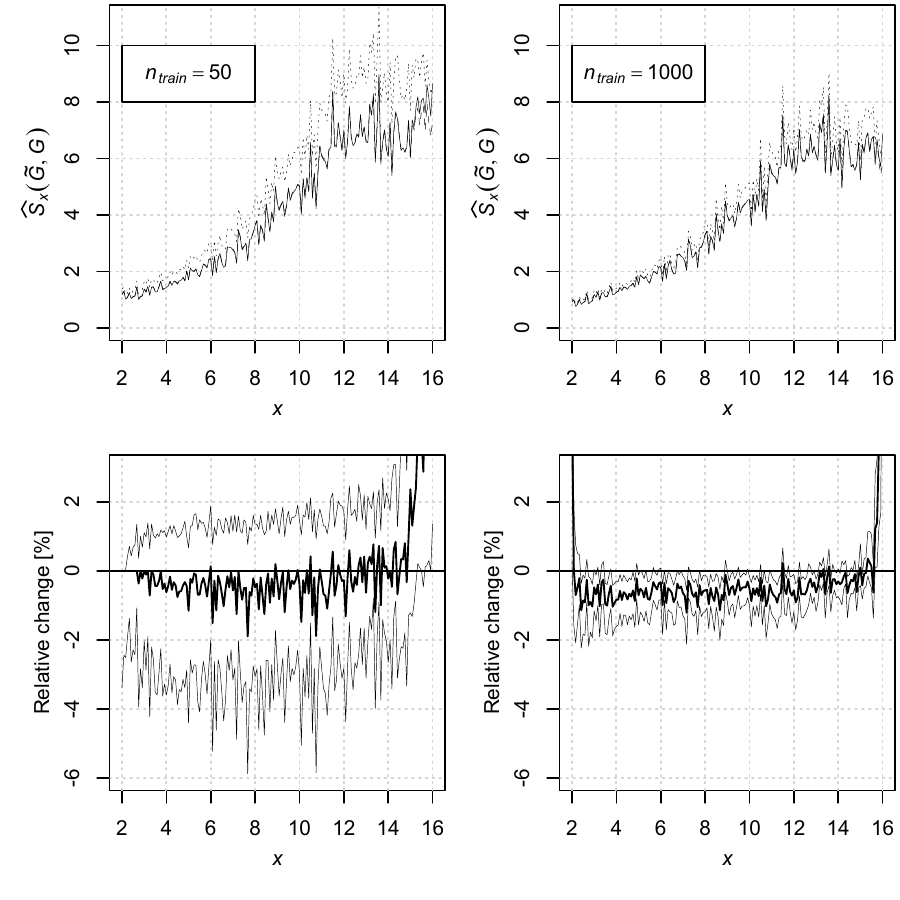}
\caption{Monte Carlo simulations to evaluate prediction performances using an empirical CRPS score. First row: empirical CRPS scores with $\tilde{G}$ being either $\hat{G}$ (solid line), $\check{G}$ (dashed line, hardly distinguishable from solid line) or $\hat{\mathbb{G}}$ (dotted line). Second row: Relative change of score when enforcing a likelihood ratio order constraint over the usual stochastic order constraint.}
\label{Fig:Growth_CRPS}
\end{figure}

Figure~\ref{Fig:Growth_CRPS} shows the forecast evaluation results. As expected, the empirical CRPS increases with age, since the spread of the weight increases with age. As to the relative change in score, improvements of about $0.5\%$ can be seen for both training sample sizes. The region of $\mathfrak{X}$ where the estimator under likelihood ratio order constraint shows better predictive performances is the widest for the largest training sample size. These results show the benefit of a stronger regularization.


\section*{Code availability}
Our procedure is implemented in the \proglang{R}-package \pkg{LRDistReg} and is available from the \proglang{GitHub} of the first author: \url{https://github.com/AlexandreMoesching/LRDistReg}. Its implementation includes \proglang{C++} code which is then integrated in \proglang{R} using \pkg{Rcpp}.


%
\appendix

\section{Proofs and technical details}
\label{App:Estimation}
%
\subsection{Proofs for Sections~\ref{Sec:EL} and \ref{Sec:Estimation}}
%
\begin{Lemma}
\label{Lem:Lin.interpolation.lr}
Let $Q_0$ and $Q_1$ be probability distributions on $\R$ such that $Q_0 \lelr Q_1$. If we define $Q_t := (1 - t) Q_0 + t Q_1$ for $0 < t < 1$, then $Q_s \lelr Q_t$ for $0 \le s < t \le 1$.
\end{Lemma}

\begin{proof}[\textbf{Proof}]
By assumption, there exist densitites $g_0$ of $Q_0$ and $g_1$ of $Q_1$ with respect to some dominating measure $\mu$ such that $g_1/g_0$ is isotonic on $\{g_0 + g_1 > 0\}$, and this is equivalent to the property that
\[
	g_0(y) g_1(x) \ \le \ g_0(x) g_1(y)
	\quad\text{whenever} \ x < y .
\]
Now, $Q_t$ has density $g_t := (1 - t) g_0 + t g_1$ with respect to $\mu$, and elementary algebra reveals that for $0 \le s < t \le 1$ and arbitrary $x < y$,
\[
	g_s(x) g_t(y) - g_s(y) g_t(x) \ = \ (t - s) \bigl( g_0(x) g_1(y) - g_0(y) g_1(x) \bigr)
	\ \ge \ 0 ,
\]
whence $Q_s \lelr Q_t$.
\end{proof}

\begin{proof}[\textbf{Proof of Lemma~\ref{Lem:Dim.reduction}}]
Let $\bs{h} \in [0,\infty)$ satisfy \eqref{ineq:EL} and $L(\bs{h}) > - \infty$.

As for part~(a), it follows from $L(\bs{h}) > - \infty$ that $h_{jk} > 0$ whenever $w_{jk} > 0$. We have to show that for arbitrary index pairs $(j_1,k_2), (j_2,k_1)$ with $j_1 \le j_2$, $k_1 \le k_2$ and $w_{j_1k_2}, w_{j_2k_1} > 0$, also $h_{jk} > 0$ for all $j \in \{j_1,\ldots,j_2\}$ and $k \in \{k_1,\ldots,k_2\}$.

Since $h_{j_1k_2}, h_{j_2k_1} > 0$, it follows from \eqref{ineq:EL} that $h_{j_1k_1}, h_{j_2k_2} > 0$, too. (If $j_1 = j_2$ or $k_1 = k_2$, this conclusion is trivial.) This type of argument will reappear several times, so we denote it by $A(j_1,j_2,k_1,k_2)$.

Next we show that $h_{jk_1}, h_{jk_2} > 0$ for $j_1 < j < j_2$. Indeed, there exists an index $k_*$ such that $w_{jk_*} > 0$, whence $h_{jk_*} > 0$. If $k_* \le k_2$, we may conclude from $A(j_1,j,k_*,k_2)$ that $h_{j,k_2} > 0$, and then it follows from $A(j,j_2,k_1,k_2)$ that $h_{jk_1} > 0$. Similarly, if $k_* \ge k_1$, we may conclude from $A(j,j_2,k_1,k_*)$ that $h_{jk_1} > 0$, and then $A(j_1,j,k_1,k_2)$ shows that $h_{jk_2} > 0$.

Analogously, one can show that $h_{j_1k}, h_{j_2k} > 0$ for $k_1 < k < k_2$.

Finally, if $j_1 < j < j_2$ and $k_1 < k < k_2$, then we may apply $A(j_1,j,k_1,k)$ or $A(j,j_2,k,k_2)$ to deduce that $h_{jk} > 0$.

As to part~(b), since $\PP$ contains all pairs $(j,k)$ with $w_{jk} > 0$, we know that $L_{\rm raw}(\tilde{\bs{h}}) = L_{\rm raw}(\bs{h})$, and $n - n \tilde{h}_{++} \ge n - n h_{++}$ with equality if and only if $\tilde{\bs{h}} = \bs{h}$. This proves the assertions about $L(\tilde{\bs{h}})$ and $L(\bs{h})$. That $\tilde{\bs{h}}$ inherits property~\eqref{ineq:EL} from $\bs{h}$ can be deduced from the fact that for indices $j_1 < j_2$ and $k_1 < k_2$, it follows from $\tilde{h}_{j_1k_2} \tilde{h}_{j_2k_1} > 0$, that $(j_1,k_2), (j_2,k_1) \in \PP$, so $(j_1,k_1), (j_2,k_2) \in \PP$ as well, and $\tilde{h}_{j_1k_1} \tilde{h}_{j_2k_2} - \tilde{h}_{j_1k_2} \tilde{h}_{j_2k_1}$ is identical to $h_{j_1k_1} h_{j_2k_2} - h_{j_1k_2} h_{j_2k_1} \ge 0$.

Concerning part~(c), we have to show that \eqref{ineq:EL2} implies \eqref{ineq:EL}. To this end, let $(j_1,k_2), (j_2,k_1) \in \PP$ with $j_1 < j_2$ and $k_1 < k_2$. Since $\{j_1,\ldots,j_2\} \times \{k_1,\ldots,k_2\} \subset \PP$, one can write
\[
	\frac{h_{j-1,k_1} \, h_{j,k_2}}{h_{j-1,k_2} \, h_{j,k_1}} \
	= \ \prod_{k=k_1+1}^{k_2}
		\frac{h_{j-1,k-1} \, h_{j,k}}{h_{j-1,k} \, h_{j,k-1}} \
	\ge \ 1
\]
for $j_1 < j \le j_2$, and
\[
	\frac{h_{j_1,k_1} \, h_{j_2,k_2}}{h_{j_1,k_2} \, h_{j_2,k_1}} \
	= \ \prod_{j=j_1+1}^{j_2}
		\frac{h_{j-1,k_1} \, h_{j,k_2}}{h_{j-1,k_2} \, h_{j,k_1}} \
	\ge \ 1 ,
\]
so \eqref{ineq:EL} is satisfied as well. 
\end{proof}

\begin{proof}[\textbf{Proof of Theorem~\ref{Thm:Existence.uniqueness}}]
Since $f$ is strictly convex and $\Theta$ is convex, $f$ has at most one minimizer in $\Theta$. To prove existence of a minimizer, it suffices to show that
\begin{equation}
\label{eq:coercivity}
	f(\thb) \ \to \ \infty \quad\text{as} \
	\thb \in \Theta, \|\thb\| \to \infty .
\end{equation}
Suppose that \eqref{eq:coercivity} is false. Then there exists a sequence $(\thb^{(s)})_s$ in $\Theta$ such that $\|\thb\| \to \infty$ but $\bigl( f(\thb^{(s)}) \bigr)_s$ is bounded. With $r_s := \|\thb^{(s)}\|$ and $\u^{(s)} := r_s^{-1} \thb^{(s)}$, we may assume without loss of generality that $\u^{(s)} \to \u$ as $s \to \infty$ for some $\u \in \Theta$ with $\|\u\| = 1$. For any fixed $t > 0$ and sufficiently large $s$, convexity and differentiablity of $f$ imply that
\begin{align*}
	f(\thb^{(s)}) \
	&= \ f(t \u^{(s)}) + \bigl( f(r_s \u^{(s)}) - f(t \u^{(s)}) \bigr) \\
	&\ge \ f(t \u^{(s)}) + (r_s - t) \partial f(t \u^{(s)}) / \partial t .
\end{align*}
Since $\lim_{s \to \infty} f(t \u^{(s)}) = f(t \u)$ and $\lim_{s \to \infty} \partial f(t \u^{(s)}) / \partial t = \partial f(t \u) / \partial t$, we conclude that
\[
	\partial f(t \u) / \partial t \ \le \ 0
	\quad\text{for all} \ t > 0 .
\]
But as $t \to \infty$, the directional derivative $\partial f(t \u) / \partial t = \sum_{(j,k) \in \PP} \bigl( - w_{jk} u_{jk} + u_{jk} \exp(t u_{jk}) \bigr)$ converges to
\[
	\begin{cases}
		\infty
			& \text{if} \ u_{jk} > 0 \ \text{for some} \ (j,k) \in \PP , \\
		\displaystyle
		- \sum_{(j,k) \in \PP} w_{jk} u_{jk}
			& \text{if} \ \u \in (-\infty,0]^{\PP} .
	\end{cases}
\]
Consequently, the limiting direction $\u$ lies in $\Theta \cap (-\infty,0]^{\PP}$ and satisfies $u_{jk} = 0$ whenever $w_{jk} > 0$. But as shown below, this implies that $\u = \bs{0}$, a contradiction to $\|\u\| = 1$.

The proof of $\bs{u} = \bs{0}$ is very similar to the proof of Lemma~\ref{Lem:Dim.reduction}. If $j_1 \le j_2$ and $k_1 \le k_2$ are indices such that $u_{j_1k_2} = u_{j_2k_1} = 0$, then it follows from $\u \in (-\infty,0]^{\PP}$ and \eqref{ineq:ELL2} that $u_{j_1k_1} + u_{j_2k_2} \ge 0$, whence $u_{j_1k_1} = u_{j_2k_2} = 0$. Repeating this argument as in the proof of Lemma~\ref{Lem:Dim.reduction}, one can show that for arbitrary $(j_1,k_2), (j_2,k_1) \in \PP$ with $j_1 \le j_2$, $k_1 \le k_2$, and $w_{j_1k_2}, w_{j_2,k_1} > 0$, we have $u_{jk} = 0$ for $j_1 \le j \le j_2$ and $k_1 \le k \le k_2$. By definition of $\PP$, this means that $\u = \bs{0}$.
\end{proof}

\begin{proof}[\textbf{Proof of Lemma~\ref{Lem:Properties.Psi}}]
With the linear bijection $T : \R^{\PP} \to \R^{\PP}$ and $\tilde{\Theta} = T(\Theta)$, $\thbtil = T(\thb)$, $\tilde{f} = f \circ T^{-1}$, one can show that for arbitrary $\x \in \R^{\PP}$ and $\xtil = T(\x)$,
\[
	\bigl\langle \nabla \tilde{f}(\thbtil), \xtil - \thbtil \bigr\rangle
	\ = \ \bigl\langle \nabla f(\thb), \x - \thb \bigr\rangle ,
\]
so
\[
	\Psi(\thb) \ = \ \argmin_{\x \in \Theta}
		\bigl( f(\thb) + \bigl\langle \nabla f(\thb), \x - \thb\bigr\rangle
			+ 2^{-1} \bigl\| \A_{\thb}(\x) - \A_{\thb}(\thb) \bigr\|^2 \bigr)
\]
with
\[
	\A_{\thb}(\x) \ := \ \bigl( \tilde{v}_{jk}(\thb)^{1/2} T_{jk}(\x) \bigr)_{(j,k) \in \PP}
\]
and $\tilde{v}_{jk}(\thb) := \partial^2 \tilde{f}(\thbtil)/\partial \tilde{\theta}_{jk}^2$. It follows from parts~(i) and (ii) of Lemma~\ref{Lem:Quadratic.proxy} in Section~\ref{Subsec:Quadratic.proxy} that $\Psi$ is continuous on $\R^{\PP}$, and that $\delta(\thb) = \bigl\langle \nabla f(\thb), \thb - \Psi(\thb) \bigr\rangle > 0$ for $\thb \in \Theta \setminus \{\hat{\thb}\}$. Moreover,
\[
	f(\thb) - f(\hat{\thb})
	\ \le \ \max \Bigl( 2 \delta(\thb),
		\sqrt{2\delta(\thb)} \, \|\A_{\thb}(\thb - \hat{\thb})\| \Bigr) .
\]
But
\[
	\|\A_{\thb}(\x)\|^2
	\ \le \ \max_{(j,k) \in \PP} \, \tilde{v}_{jk}(\thb) \|T(\x)\|^2
	\ \le \ 3 \max_{(j,k) \in \PP} \, \tilde{v}_{jk}(\thb) \|\x\|^2 ,
\]
so
\[
	f(\thb) - f(\hat{\thb})
	\ \le \ \max \Bigl( 2 \delta(\thb), \beta_1(\thb)
		\sqrt{\delta(\thb)} \, \|\thb - \hat{\thb}\| \Bigr)
\]
with $\beta_1(\thb)$ being the square root of $6 \max_{(j,k) \in \PP} \, \tilde{v}_{jk}(\thb)$. In case of $\Psi = \Psi^{\rm row}$ and $\thb$ being row-wise calibrated, $\beta_1(\thb)^2$ is no larger than $6 \max_{1 \le j \le \ell} w_{j+}$, and in case of $\Psi = \Psi^{\rm col}$ and $\thb$ being column-wise calibrated, $\beta_1(\thb)^2 \le 6 \max_{1 \le k \le m} w_{+k}$.

Concerning the lower bound for the maximum of $f(\thb) - f \bigl( (1 - t)\thb + t \Psi(\thb) \bigr)$ over all $t \in [0,1]$, note that for arbitrary $\thb', \thb'' \in \R^{\PP}$,
\begin{align*}
	\frac{\mathrm{d}^2}{\mathrm{d}t^2} f((1 - t)\thb' + t \thb'') \
	&= \ n \sum_{(j,k) \in \PP} \exp((1 - t) \theta_{jk}' + t \theta_{jk}'')
		(\theta_{jk}' - \theta_{jk}'')^2 \\
	&\le \ n \max_{(j,k) \in \PP} \exp \bigl( \max(\theta_{jk}',\theta_{jk}'') \bigr)
		\|\thb' - \thb''\|^2 .
\end{align*}
Thus part~(iii) of Lemma~\ref{Lem:Quadratic.proxy} yieds the asserted lower bound with
\[
	\beta_2(\thb) \ := \ 2 n \max_{(j,k) \in \PP} \,
		\exp \bigl( \max(\theta_{jk}, \Psi_{jk}(\thb)) \bigr) .
		\qedhere
\]
\end{proof}

\begin{proof}[\textbf{Proof of Theorem~\ref{Thm:Convergence}}]
It follows from Lemma~\ref{Lem:Properties.Psi} and the construction of the sequence $(\thb^{(s)})_{s \ge 0}$ that
\[
	f(\thb^{(s)}) - f(\thb^{(s+1)}) \ \ge \ \beta(\thb^{(s)})
\]
for all $s \ge 0$ with some continuous function $\beta : \Theta \to [0,\infty)$ such that $\beta > 0$ on $\Theta \setminus \{\hat{\thb}\}$. Note that $f(\thb^{(s)})$ is antitonic in $s \ge 0$, so the sequence $(\thb^{(s)})_{s \ge 0}$ stays in the compact set $R_0 := \bigl\{ \thb \in \Theta : f(\thb) \le f(\thb^{(0)}) \bigr\}$. For each $\thb \in R_0 \setminus \{\hat{\thb}\}$, there exists a $\delta_{\thb} > 0$ such that the open ball $U(\thb,\delta_{\thb})$ with center $\thb$ and radius $\delta_{\thb}$ satisfies
\[
	|f - f(\thb)| < \beta(\thb)/3
	\ \ \text{and} \ \
	\beta > 2 \beta(\thb)/3
	\ \ \text{on} \ U(\thb,\delta_{\thb}) .
\]
In particular, if $\thb^{(s)} \in U(\thb,\delta_{\thb})$ for some $s \ge 0$, then $f(\thb^{(s+1)}) < f(\thb) - \beta(\thb)/3$. Consequently, $\thb^{(s)} \in U(\thb,\delta_{\thb})$ for at most one index $s \ge 0$. But for each $\epsilon > 0$, the compact set $\bigl\{ \thb \in R_0 : \|\thb - \hat{\thb}\| \ge \epsilon \bigr\}$ can be covered by finitely many of these balls $U(\thb,\delta_{\thb})$. Hence, $\|\thb^{(s)} - \hat{\thb}\| \ge \epsilon$ for at most finitely many indices $s \ge 0$.
\end{proof}

\subsection{Minimizing convex functions via quadratic approximations}
\label{Subsec:Quadratic.proxy}
%
Let $f : \R^d \to \R$ be a strictly convex and differentiable function, and let $\Theta \subset \R^d$ be a closed, convex set such that a minimizer
\[
	\hat{\thb} \ := \ \argmin_{\thb \in \Theta} \, f(\thb)
\]
exists. For $\thb_o \in \Theta$ and some nonsingular matrix $\A \in \R^{d\times d}$ consider the quadratic approximation
\[
	f_o(\x) \ := \ f(\thb_o) + \nabla f(\thb_o)^\top (\x - \thb_o)
		+ 2^{-1} \|\A\x - \A\thb_o\|^2
\]
of $f(\x)$. By construction, $f_o(\thb_o) = f(\thb_o)$ and $\nabla f_o(\thb_o) = \nabla f(\thb_o)$, and there exists a unique minimizer
\[
	\thb_* \ := \ \argmin_{\thb \in \Theta} f_o(\thb) .
\]
The next lemma clarifies some connections between $\thb_*$ and $\hat{\thb}$ in terms of the directional derivative
\[
	\delta_o \ := \ \nabla f(\thb_o)^\top (\thb_o - \thb_*)
	\ = \ - \frac{\mathrm{d}}{\mathrm{d}t} \Big|_{t=0} f(\thb_o + t (\thb_* - \thb_o)) .
\]

\begin{Lemma}
\label{Lem:Quadratic.proxy}
\textbf{(i)} \ The point $\thb_*$ equals $\thb_o$ if and only if $\thb_o = \hat{\thb}$. Furthermore,
\[
	2^{-1} \delta_o \ \le \ f_o(\thb_o) - f_o(\thb_*) \ \le \ \delta_o
\]
and
\[
	f(\thb_o) - f(\hat{\thb}) \
	\le \ \nabla f(\thb_o)^\top (\thb_o - \hat{\thb}) \
	\le \ \max \Bigl( 2 \delta_o,
			\sqrt{ 2 \delta_o} \, \|\A\hat{\thb} - \A\thb_o\| \Bigr) .
\]
\textbf{(ii)} \ If $f$ is continuously differentiable, the minimizer $\thb_*$ is a continuous function of $\thb_o \in \Theta$ and $\A$.\\[0.5ex]
\textbf{(iii)} \ If $f$ is even twice differentiable such that for some constant $c_o > 0$ and any $t \in [0,1]$,
\[
	\frac{\mathrm{d}^2}{\mathrm{d} t^2} \, f((1 - t)\thb_o + t \thb_*) \
	\le \ c_o \|\thb_o - \thb_*\|^2 ,
\]
then in case of $\thb_o \ne \hat{\thb}$,
\[
	\max_{t \in [0,1]} \, \bigl( f(\thb_o) - f((1 - t) \thb_o + t \thb_*) \bigr)
	\ \ge \
	2^{-1} \min \Bigl( \delta_o,
		\frac{\delta_o^2}{c_o \|\thb_o - \thb_*\|^2} \Bigr) .
\]
\end{Lemma}

\begin{proof}[\textbf{Proof}]
By strict convexity of $f$, $\thb_o = \hat{\thb}$ if and only if
\[
	\frac{\mathrm{d}}{\mathrm{d}t} \Bigl|_{t=0} f(\thb_o + t(\thb - \thb_o))
	= \nabla f(\thb_o)^\top (\thb - \thb_o) \ \ge \ 0
	\quad\text{for all} \ \thb \in \Theta .
\]
But since $f_o$ is strictly convex, too, with $\nabla f_o(\thb_o) = \nabla f(\thb_o)$, the latter displayed condition is also equivalent to $\thb_o = \thb_*$.

Since the asserted inequalities are trivial in case of $\thb_o = \hat{\thb} = \thb_*$, let us assume in the sequel that $\thb_* \ne \thb_o \ne \hat{\thb}$. By convexity of $f$ and $f_o$,
\[
	f_o(\thb_o) - f_o(\thb_*)
	\ \le \ \frac{\mathrm{d}}{\mathrm{d}t} \Big|_{t = 1} f_o(\thb_* + t(\thb_o - \thb_*))
	\ = \ \delta_o
\]
and
\[
	f(\thb_o) - f(\hat{\thb})
	\ \le \ \frac{\mathrm{d}}{\mathrm{d}t} \Big|_{t = 1} f(\hat{\thb} + t(\thb_o - \hat{\thb}))
	\ = \ \nabla f(\thb_o)^\top (\thb_o - \hat{\thb}) .
\]
On the other hand, since $\thb_*$ minimizes $f_o$ over $\Theta$,
\[
	0 \ \le \ \frac{\mathrm{d}}{\mathrm{d}t} \Big|_{t=0} f_o(\thb_* + t(\thb_o - \thb_*))
	\ = \ \nabla f_o(\thb_*)^\top (\thb_o - \thb_*)
	\ = \ \delta_o - \|\A\thb_o - \A\thb_*\|^2 ,
\]
so
\[
	f_o(\thb_o) - f_o(\thb_*)
	\ = \ \delta_o - 2^{-1} \|\A\thb_o - \A\thb_*\|^2
	\ \ge \ 2^{-1} \delta_o .
\]
Moreover, with $\hat{\delta} := \nabla f(\thb_o)^\top (\thb_o - \hat{\thb})$ and $\hat{\gamma} := \|\A\thb_o - \A\hat{\thb}\|^2$,
\begin{align*}
	2 \delta_o \ \ge \
	2 \bigl( f_o(\thb_o) - f_o(\thb_*) \bigr) \
	&= \ 2 \max_{\thb \in \Theta} \,
		\bigl( f_o(\thb_o) - f_o(\thb) \bigr) \\
	&\ge \ 2 \max_{t \in [0,1]} \,
		\bigl( f_o(\thb_o) - f_o((1 - t) \thb_o + t \hat{\thb}) \bigr) \\
	&= \ \max_{t \in [0,1]} \,
		\bigl( 2 t \hat{\delta} - t^2 \hat{\gamma} \bigr) \\
	&= \ 2 t_o \hat{\delta}
		- t_o^2 \hat{\gamma} ,
\end{align*}
where $t_o := \min \bigl( 1, \hat{\delta}/\hat{\gamma} \bigr)$.
In case of $\hat{\delta} \ge \hat{\gamma}$, we may conclude that $2\delta_o \ge 2\hat{\delta} - \hat{\gamma} \ge \hat{\delta}$, so $\hat{\delta} \le 2 \delta_o$, and otherwise, $2 \delta_o \ge \hat{\delta}^2 / \hat{\gamma}$, whence $\hat{\delta} \le \sqrt{2 \delta_o \hat{\gamma}}$. This proves part~(i).

As to part~(ii), let $(\thb_o^{(s)})_{s \ge 1}$ be a sequence in $\Theta$ with limit $\thb_o$, and let $(\A^{(s)})_{s \ge 1}$ be a sequence of nonsingular matrices in $\R^{d\times d}$ converging to a nonsingular matrix $\A$. Definining $f_o^{(s)}$ as $f_o$ with $(\thb_o^{(s)}, \A^{(s)})$ in place of $(\thb,\A)$, we know that $f_o^{(s)} \to f_o$ as $s \to \infty$ uniformly on any bounded subset of $\R^d$. Consequently, for any fixed $\epsilon > 0$ and $R_\epsilon := \bigl\{ \thb \in \Theta : \|\thb - \thb_*\| = \epsilon \bigr\}$,
\[
	\gamma_\epsilon^{(s)}
		:= \min_{\thb \in R_\epsilon} \, f_o^{(s)}(\thb) - f_o^{(s)}(\thb_*)
	\ \to \
	\gamma_\epsilon
		:= \min_{\thb \in R_\epsilon} \, f_o(\thb) - f_o(\thb_*) > 0
\]
as $s \to \infty$. But as soon as $\gamma_\epsilon^{(s)} > 0$, it follows from convexity of $\Theta$ and $f^{(s)}$ that the minimizer $\thb_*^{(s)}$ of $f_o^{(s)}$ satisfies $\|\thb_*^{(s)} - \thb_*\| < \epsilon$.

Part~(iii) follows from
\begin{align*}
	\max_{t \in [0,1]} \,
		\bigl( f(\thb_o) - f((1 - t)\thb_o + t \thb_*) \bigr) \
	&= \ \max_{t \in [0,1]} \,
		\bigl( f(\thb_o) - f(\thb_o + t (\thb_* - \thb_o)) \bigr) \\
	&\ge \ \max_{t \in [0,1]} \bigl( t \delta_o
		- 2^{-1} t^2 c_o \|\thb_o - \thb_*\|^2 \bigr) \\
	&= \ t_o \delta_o - 2^{-1} t_o^2 c_o \|\thb_o - \thb_*\|^2 \\
	&\ge \ 2^{-1} \min \Bigl( \delta_o,
		\frac{\delta_o^2}{c_o \|\thb_o - \thb_*\|^2} \Bigr) ,
\end{align*}
where $t_o := \min \bigl( 1, \delta_o / (c_o \|\thb_o - \thb_*\|^2) \bigr)$.
\end{proof}

\subsection{Technical details for Sections~\ref{Sec:Simulation.study}}
\label{Sec:RewritingScores}
%
For fixed $\ell_o,n\in\N$, let $(\hat{G}_x)_{x\in\mathfrak{X}_o}$, $(\check{G}_x)_{x\in\mathfrak{X}_o}$ and $(\hat{\mathbb{G}}_x)_{x\in\mathfrak{X}_o}$ be estimates of $(G_x)_{x\in\mathfrak{X}_o}$ from a sample $\{(X_i,Y_i)\}_{i=1}^n$ as described in Section~\ref{Sec:Sampling}. Then, for all $\tilde{G}\in \{\hat{G}, \check{G}, \hat{\mathbb{G}}\}$ and $x\in\mathfrak{X}_o$, the estimate $\tilde{G}_{x}$ is a step function with jumps in the set $\{y_1,\ldots,y_m\}$ of unique observations. For convenience, we further denote $y_0:=0$, $y_{m+1}:=\infty$, and define
\[
	\tilde{G}_{jk} \ := \ \tilde{G}_{x_j}(y_k),
	\quad 0 \le k \le m,
	\qquad \text{and} \qquad
	\tilde{G}_{jm+1} \ := \ 1,
\]
for all $1\le j \le \ell_o$ and $\tilde{G}\in \{G, \hat{G}, \check{G}, \hat{\mathbb{G}}\}$. 

For the remainder of this section, we fix $1\le j \le \ell_o$ and $\tilde{G}\in \{\hat{G}, \check{G}, \hat{\mathbb{G}}\}$. Observe that $R_{x_j}(\tilde{G}, G)$ is the sum of the terms
\[
	R_{x_j}^{(k)}(\tilde{G}, G)
	\ = \ 
	\int_{y_k}^{y_{k+1}} 
		| \tilde{G}_{jk} - G_{x_j}(y) | \, 
		g_{x_j}(y)
		\,\mathrm{d}y,
\]
defined for $0 \le k \le m$, where $g_{x_j}$ is the density of $Q_{x_j}$ with respect to Lebesgue measure. But since
\[
	\int_\alpha^\beta G_{x_j}(y) g_{x_j}(y) \,\mathrm{d}y
	\ = \
	\frac{G_{x_j}(\beta)^2 - G_{x_j}(\alpha)^2}{2},
\]
we find that
\begin{align*}
	R_{x_j}^{(0)}(\tilde{G}, G)
	\ &= \
	G_{j1}^2/2, \\
	R_{x_j}^{(k)}(\tilde{G}, G)
	\ &= \
	\begin{cases}
		\rho( \tilde{G}_{jk}, G_{jk+1} ) -
		\rho( \tilde{G}_{jk}, G_{jk  } )
		& \text{if}\ \tilde{G}_{jk} \ge G_{jk+1}, \\
		\rho( \tilde{G}_{jk}, G_{jk  } ) -
		\rho( \tilde{G}_{jk}, G_{jk+1} )
		& \text{if}\ \tilde{G}_{jk} \le G_{jk}, \\		
		\tilde{G}_{jk}^2 -		
		\rho( \tilde{G}_{jk}, G_{jk  } ) -
		\rho( \tilde{G}_{jk}, G_{jk+1} )
		& \text{otherwise,}
	\end{cases}\\
	R_{x_j}^{(m)}(\tilde{G}, G)
	\ &= \
	1/2 - \rho(1, G_{jm}),
\end{align*}
for $1\le k < m$, where $\rho(z_1, z_2) \ := \ z_1z_2 - z_2^2/2$.

Similarly, the computation of the CRPS involves the sum of the following integrals
\[
	S_{x_j}^{(k)}(\tilde{G}, G)
	\ := \
	\int_{y_k}^{y_{k+1}} 
		\bigl(\tilde{G}_{jk} - G_{x_j}(y) \bigr)^2 \, \mathrm{d}y,
\]
defined for $0 \le k \le m$. But integration by parts yields
\[
	\int_{\alpha}^\beta G_{x_j}(y)\, \mathrm{d}y
	\ = \
	\beta G_{x_j}(\beta) - \alpha G_{x_j}(\alpha) 
	- c_j
	\bigl(
		\bar{G}_{x_j}(\beta) - \bar{G}_{x_j}(\alpha)
	\bigr)
\]
where $c_j := b(x_j)\Gamma(a(x_j)+1)/\Gamma(a(x_j))$ and $\bar{G}_{x_j}$ denotes the cumulative distribution function of a Gamma distribution with shape $a(x_j)+1$ and scale $b(x_j)$. In consequence, if we define $\bar{G}_{jk}:= \bar{G}_{x_j}(y_k)$ and
\[
	I_{x_j}^{(k)} \ := \ 
	\tilde{G}_{jk}^2 (y_{k+1} - y_k)
	- 2 \tilde{G}_{jk} \Bigl(
		y_{k+1} G_{jk+1} - y_k G_{jk}
		- c_j
		( \bar{G}_{jk+1} - \bar{G}_{jk} )
	\Bigr)
\]
for $1 \le k < m$, we obtain
\begin{align*}
	\sum_{k=0}^m S_{x_j}^{(k)}(\tilde{G}, G)
	\ = \
	\int_{0}^{y_m} G_{x_j}(y)^2 \, \mathrm{d}y +
	\int_{y_m}^{\infty} \bigl(1 - G_{x_j}(y)\bigr)^2 \, \mathrm{d}y +
	\sum_{k=1}^{m-1} I_{x_j}^{(k)},
\end{align*}
where the above two integrals are computed numerically.
\end{document}